\documentclass{amsart}

\usepackage{amsmath,amsfonts,amssymb,amsthm}
\usepackage{tikz}
\usepackage{caption}
\usepackage{float}

\newtheorem{theorem}{Theorem}[section]

\newtheorem{lemma}[theorem]{Lemma}

\theoremstyle{definition}

\newtheorem{example}[theorem]{Example}

\theoremstyle{remark}

\newtheorem{remark}[theorem]{Remark}

\numberwithin{equation}{section}

\begin{document}

\title{Confluence in Labeled Chip-Firing}

\author{Caroline Klivans, Patrick Liscio}

\begin{abstract}

  In 2016, Hopkins, McConville, and Propp proved that labeled
  chip-firing on a line always leaves the chips in sorted order if the
  number of chips is even.  We present a novel proof of this
  result.  We then apply our methods to resolve a number of related
  conjectures concerning the confluence of labeled chip-firing
  systems.

\end{abstract}

\maketitle

\section{Introduction}

This paper is concerned with a labeled variant of the chip-firing
process as defined by Hopkins, McConville, and Propp \cite{HMP16}.  In
unlabeled chip-firing, a collection of indistinguishable chips are
placed at the nodes of a graph.  If a node has at least as many chips
as it has neighbors, it can ``fire'' by sending one chip to each of
its neighbors.  The process terminates if no site has enough chips to
fire.  In the labeled chip-firing process a collection of labeled
chips are placed at the origin of the infinite path graph.  Again, if a node
has at least as many chips as it has neighbors, it can fire.  The difference in the labeled case is that  two distinct 
chips at the node are chosen and the chip of larger label
moves to the right while the chip of smaller label moves to the left,
see Example~\ref{ex1}.

In \cite{HMP16}, it was shown that, for an even number of chips,
labeled chip-firing terminates in a unique configuration regardless of the order in which nodes fire and regardless of 
the choice of chips made at each node.  Moreover, in the unique terminal configuration, the chips
are in sorted order.  The property that the final configuration is
unique regardless of the intermediate steps is known as global confluence
and is a fundamental property of unlabeled chip-firing.  In unlabeled
chip-firing global confluence is proved via local confluence.  Local confluence states that for any two available fires, there is a
common configuration that can be reached after either of them in only
one additional fire.  In the case of unlabeled chip-firing, any two
available fires may be performed in either order without changing the
resulting configuration.  Local confluence combined with Newman's
Lemma on abstract rewriting systems \cite{Newman42} gives a global
confluence property for unlabeled chip-firing, in which any
terminating chip-firing process must have a {unique} final
configuration.

In labeled chip-firing, local confluence does not hold.
The sorting result for labeled chip-firing from \cite{HMP16} is
particularly notable for proving global confluence even though local
confluence, and thus Newman's Lemma does not apply.  Without this
tool, the labeled case proved significantly more challenging to
establish.  

Here we give a novel, more general and more illuminating proof of
global confluence for labeled chip-firing and related systems.  Our
proof is based on the analysis of a firing order poset.
Figures~\ref{fign10} and~\ref{fign20} visually demonstrate the
structure of confluent versus non-confluent labeled chip-firing processes that arise if we
have an even or odd number of chips initially.  The existence of the
diamond shape at the bottom of the Haase diagram in the even case is crucial for confluence.  

There have been attempts to generalize the results of \cite{HMP16} including several conjectures from \cite{Hopkins16} and \cite{HMP16} in which labeled chip-firing is extended to modified versions of the 1-dimensional grid graph.  Additionally, Galashin et al. \cite{GHMP2, GHMP1}  treat the labeled chip-firing problem as chip-firing on a Type A root system, and then generalize the problem to apply to other types of root systems and 
more general classes of firing moves.  In \cite{FK19}, global confluence is proved without local confluence for higher dimension forms of chip-firing.  
Our methods allows us to prove many of the above cases via a unified methodology.  
  
In Section~\ref{main}, we present the proof that labeled chip-firing sorts.
In Section~\ref{related}, we
apply these methods to prove a series of related conjectures on
sorting via chip-firing on modified versions of the one-dimensional
grid graph.  In Section~\ref{non-sorting}, we discuss how these methods can shed light on 
 the case where the number of chips is odd.

\begin{figure}
\centering
\begin{tikzpicture}[scale=0.4]



\draw[black, thick] (0,0) -- (0,-10);
\draw[black, thick] (1,-2) -- (1,-9);
\draw[black, thick] (-1,-2) -- (-1,-9);
\draw[black, thick] (2,-5) -- (2,-9);
\draw[black, thick] (-2,-5) -- (-2,-9);
\draw[black, thick] (0,-1) -- (1,-2);
\draw[black, thick] (0,-1) -- (-1,-2);
\draw[black, thick] (0,-3) -- (2,-5);
\draw[black, thick] (0,-4) -- (1,-5);
\draw[black, thick] (0,-3) -- (-2,-5);
\draw[black, thick] (0,-4) -- (-1,-5);
\draw[black, thick] (0,-6) -- (3,-9);
\draw[black, thick] (0,-7) -- (2,-9);
\draw[black, thick] (0,-8) -- (1,-9);
\draw[black, thick] (0,-6) -- (-3,-9);
\draw[black, thick] (0,-7) -- (-2,-9);
\draw[black, thick] (0,-8) -- (-1,-9);
\draw[black, thick] (0,-10) -- (4,-14);
\draw[black, thick] (-1,-11) -- (3,-15);
\draw[black, thick] (-2,-12) -- (2,-16);
\draw[black, thick] (-3,-13) -- (1,-17);
\draw[black, thick] (-4,-14) -- (0,-18);
\draw[black, thick] (0,-10) -- (-4,-14);
\draw[black, thick] (1,-11) -- (-3,-15);
\draw[black, thick] (2,-12) -- (-2,-16);
\draw[black, thick] (3,-13) -- (-1,-17);
\draw[black, thick] (4,-14) -- (0,-18);
\draw[black, thick] (1,-2) -- (0,-6);
\draw[black, thick] (1,-4) -- (0,-7);
\draw[black, thick] (1,-5) -- (0,-8);
\draw[black, thick] (1,-7) -- (0,-9);
\draw[black, thick] (1,-9) -- (0,-10);
\draw[black, thick] (-1,-2) -- (0,-6);
\draw[black, thick] (-1,-4) -- (0,-7);
\draw[black, thick] (-1,-5) -- (0,-8);
\draw[black, thick] (-1,-7) -- (0,-9);
\draw[black, thick] (-1,-9) -- (0,-10);
\draw[black, thick] (2,-8) -- (1,-9);
\draw[black, thick] (2,-9) -- (1,-11);
\draw[black, thick] (-2,-8) -- (-1,-9);
\draw[black, thick] (-2,-9) -- (-1,-11);
\draw[black, thick] (3,-9) -- (2,-12);
\draw[black, thick] (-3,-9) -- (-2,-12);

\filldraw[black] (0,0) circle (2pt);
\filldraw[black] (0,-1) circle (2pt);
\filldraw[black] (0,-2) circle (2pt);
\filldraw[black] (0,-3) circle (2pt);
\filldraw[black] (0,-4) circle (2pt);
\filldraw[black] (0,-5) circle (2pt);
\filldraw[black] (0,-6) circle (2pt);
\filldraw[black] (0,-7) circle (2pt);
\filldraw[black] (0,-8) circle (2pt);
\filldraw[black] (0,-9) circle (2pt);
\filldraw[black] (0,-10) circle (2pt);
\filldraw[black] (0,-12) circle (2pt);
\filldraw[black] (0,-14) circle (2pt);
\filldraw[black] (0,-16) circle (2pt);
\filldraw[black] (0,-18) circle (2pt);
\filldraw[black] (1,-2) circle (2pt);
\filldraw[black] (1,-4) circle (2pt);
\filldraw[black] (1,-5) circle (2pt);
\filldraw[black] (1,-7) circle (2pt);
\filldraw[black] (1,-8) circle (2pt);
\filldraw[black] (1,-9) circle (2pt);
\filldraw[black] (1,-11) circle (2pt);
\filldraw[black] (1,-13) circle (2pt);
\filldraw[black] (1,-15) circle (2pt);
\filldraw[black] (1,-17) circle (2pt);
\filldraw[black] (-1,-2) circle (2pt);
\filldraw[black] (-1,-4) circle (2pt);
\filldraw[black] (-1,-5) circle (2pt);
\filldraw[black] (-1,-7) circle (2pt);
\filldraw[black] (-1,-8) circle (2pt);
\filldraw[black] (-1,-9) circle (2pt);
\filldraw[black] (-1,-11) circle (2pt);
\filldraw[black] (-1,-13) circle (2pt);
\filldraw[black] (-1,-15) circle (2pt);
\filldraw[black] (-1,-17) circle (2pt);
\filldraw[black] (2,-5) circle (2pt);
\filldraw[black] (2,-8) circle (2pt);
\filldraw[black] (2,-9) circle (2pt);
\filldraw[black] (2,-12) circle (2pt);
\filldraw[black] (2,-14) circle (2pt);
\filldraw[black] (2,-16) circle (2pt);
\filldraw[black] (-2,-5) circle (2pt);
\filldraw[black] (-2,-8) circle (2pt);
\filldraw[black] (-2,-9) circle (2pt);
\filldraw[black] (-2,-12) circle (2pt);
\filldraw[black] (-2,-14) circle (2pt);
\filldraw[black] (-2,-16) circle (2pt);
\filldraw[black] (3,-9) circle (2pt);
\filldraw[black] (3,-13) circle (2pt);
\filldraw[black] (3,-15) circle (2pt);
\filldraw[black] (-3,-9) circle (2pt);
\filldraw[black] (-3,-13) circle (2pt);
\filldraw[black] (-3,-15) circle (2pt);
\filldraw[black] (4,-14) circle (2pt);
\filldraw[black] (-4,-14) circle (2pt);

\draw[black, thick] (10,0) -- (10,-18);
\draw[black, thick] (11,-2) -- (11,-17);
\draw[black, thick] (9,-2) -- (9,-17);
\draw[black, thick] (12,-5) -- (12,-16);
\draw[black, thick] (8,-5) -- (8,-16);
\draw[black, thick] (13,-9) -- (13,-15);
\draw[black, thick] (7,-9) -- (7,-15);
\draw[black, thick] (10,-1) -- (11,-2);
\draw[black, thick] (10,-1) -- (9,-2);
\draw[black, thick] (10,-3) -- (12,-5);
\draw[black, thick] (10,-4) -- (11,-5);
\draw[black, thick] (10,-3) -- (8,-5);
\draw[black, thick] (10,-4) -- (9,-5);
\draw[black, thick] (10,-6) -- (13,-9);
\draw[black, thick] (10,-7) -- (12,-9);
\draw[black, thick] (10,-8) -- (11,-9);
\draw[black, thick] (10,-6) -- (7,-9);
\draw[black, thick] (10,-7) -- (8,-9);
\draw[black, thick] (10,-8) -- (9,-9);
\draw[black, thick] (10,-10) -- (14,-14);
\draw[black, thick] (10,-12) -- (13,-15);
\draw[black, thick] (10,-14) -- (12,-16);
\draw[black, thick] (10,-16) -- (11,-17);
\draw[black, thick] (10,-10) -- (6,-14);
\draw[black, thick] (10,-12) -- (7,-15);
\draw[black, thick] (10,-14) -- (8,-16);
\draw[black, thick] (10,-16) -- (9,-17);
\draw[black, thick] (11,-2) -- (10,-7);
\draw[black, thick] (11,-4) -- (10,-8);
\draw[black, thick] (11,-5) -- (10,-9);
\draw[black, thick] (11,-8) -- (10,-10);
\draw[black, thick] (11,-9) -- (10,-12);
\draw[black, thick] (11,-11) -- (10,-14);
\draw[black, thick] (11,-13) -- (10,-16);
\draw[black, thick] (11,-15) -- (10,-18);
\draw[black, thick] (9,-2) -- (10,-7);
\draw[black, thick] (9,-4) -- (10,-8);
\draw[black, thick] (9,-5) -- (10,-9);
\draw[black, thick] (9,-8) -- (10,-10);
\draw[black, thick] (9,-9) -- (10,-12);
\draw[black, thick] (9,-11) -- (10,-14);
\draw[black, thick] (9,-13) -- (10,-16);
\draw[black, thick] (9,-15) -- (10,-18);
\draw[black, thick] (12,-5) -- (11,-9);
\draw[black, thick] (12,-8) -- (11,-11);
\draw[black, thick] (12,-9) -- (11,-13);
\draw[black, thick] (12,-12) -- (11,-15);
\draw[black, thick] (12,-14) -- (11,-17);
\draw[black, thick] (8,-5) -- (9,-9);
\draw[black, thick] (8,-8) -- (9,-11);
\draw[black, thick] (8,-9) -- (9,-13);
\draw[black, thick] (8,-12) -- (9,-15);
\draw[black, thick] (8,-14) -- (9,-17);
\draw[black, thick] (13,-9) -- (12,-14);
\draw[black, thick] (13,-13) -- (12,-16);
\draw[black, thick] (7,-9) -- (8,-14);
\draw[black, thick] (7,-13) -- (8,-16);

\filldraw[black] (10,0) circle (2pt);
\filldraw[black] (10,-1) circle (2pt);
\filldraw[black] (10,-2) circle (2pt);
\filldraw[black] (10,-3) circle (2pt);
\filldraw[black] (10,-4) circle (2pt);
\filldraw[black] (10,-5) circle (2pt);
\filldraw[black] (10,-6) circle (2pt);
\filldraw[black] (10,-7) circle (2pt);
\filldraw[black] (10,-8) circle (2pt);
\filldraw[black] (10,-9) circle (2pt);
\filldraw[black] (10,-10) circle (2pt);
\filldraw[black] (10,-12) circle (2pt);
\filldraw[black] (10,-14) circle (2pt);
\filldraw[black] (10,-16) circle (2pt);
\filldraw[black] (10,-18) circle (2pt);
\filldraw[black] (11,-2) circle (2pt);
\filldraw[black] (11,-4) circle (2pt);
\filldraw[black] (11,-5) circle (2pt);
\filldraw[black] (11,-7) circle (2pt);
\filldraw[black] (11,-8) circle (2pt);
\filldraw[black] (11,-9) circle (2pt);
\filldraw[black] (11,-11) circle (2pt);
\filldraw[black] (11,-13) circle (2pt);
\filldraw[black] (11,-15) circle (2pt);
\filldraw[black] (11,-17) circle (2pt);
\filldraw[black] (9,-2) circle (2pt);
\filldraw[black] (9,-4) circle (2pt);
\filldraw[black] (9,-5) circle (2pt);
\filldraw[black] (9,-7) circle (2pt);
\filldraw[black] (9,-8) circle (2pt);
\filldraw[black] (9,-9) circle (2pt);
\filldraw[black] (9,-11) circle (2pt);
\filldraw[black] (9,-13) circle (2pt);
\filldraw[black] (9,-15) circle (2pt);
\filldraw[black] (9,-17) circle (2pt);
\filldraw[black] (12,-5) circle (2pt);
\filldraw[black] (12,-8) circle (2pt);
\filldraw[black] (12,-9) circle (2pt);
\filldraw[black] (12,-12) circle (2pt);
\filldraw[black] (12,-14) circle (2pt);
\filldraw[black] (12,-16) circle (2pt);
\filldraw[black] (8,-5) circle (2pt);
\filldraw[black] (8,-8) circle (2pt);
\filldraw[black] (8,-9) circle (2pt);
\filldraw[black] (8,-12) circle (2pt);
\filldraw[black] (8,-14) circle (2pt);
\filldraw[black] (8,-16) circle (2pt);
\filldraw[black] (13,-9) circle (2pt);
\filldraw[black] (13,-13) circle (2pt);
\filldraw[black] (13,-15) circle (2pt);
\filldraw[black] (7,-9) circle (2pt);
\filldraw[black] (7,-13) circle (2pt);
\filldraw[black] (7,-15) circle (2pt);
\filldraw[black] (14,-14) circle (2pt);
\filldraw[black] (6,-14) circle (2pt);

\end{tikzpicture}
\caption{Firing order poset for $n=10$ (left) and $n=11$ (right).}
  \label{fign10}
\setlength{\belowcaptionskip}{-10pt}
\end{figure}
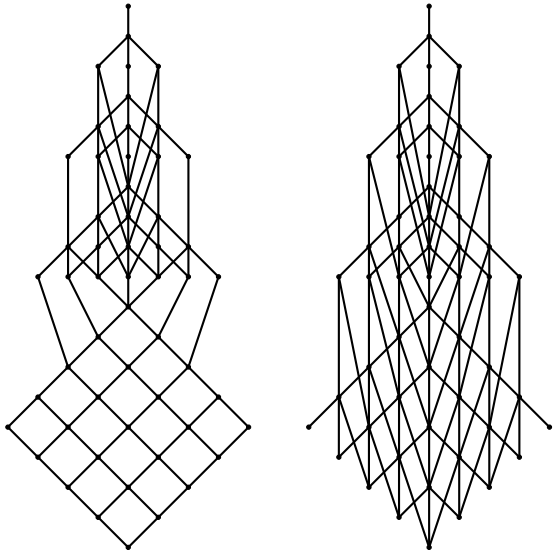

\begin{figure}
  \begin{center}
    \input{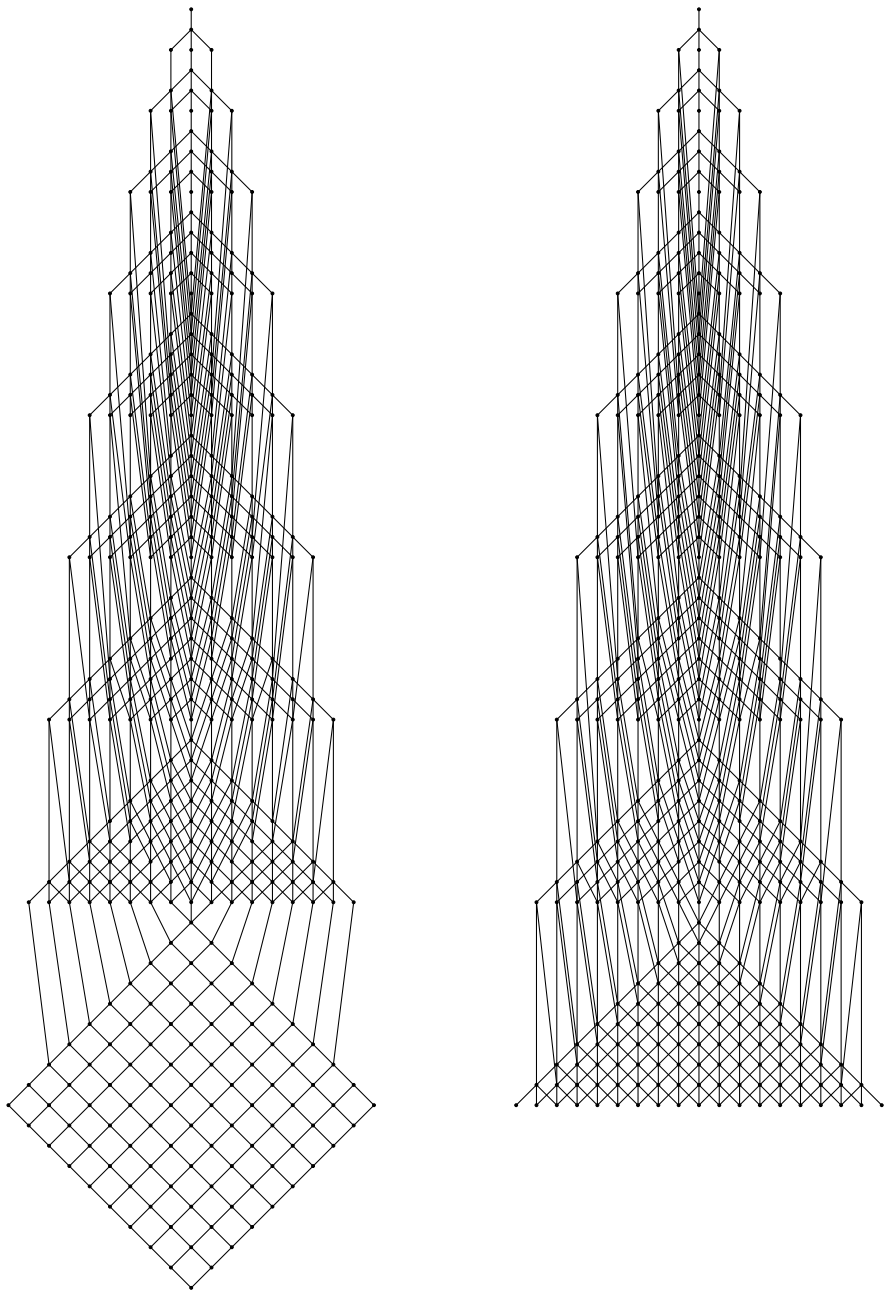}
\end{center}
\caption{Firing order poset for $n=20$ (left) and $n=21$ (right)}
  \label{fign20}
\end{figure}

\section{Sorting}\label{main}

Labeled chip-firing is defined formally as follows.   Consider the infinite path graph (or 1-dimensional grid) on $\mathbb{Z}$, where each integer $i$ is connected by a single undirected edge to both $i-1$ and $i+1$.  Place $n$ chips, labeled from $1$ to $n$, at site 0.  A firing move consists of choosing two chips labeled $a$ and $b$ ($a<b$) at a common site $i$.  Chip $a$ is sent to site $i-1$, while chip $b$ is sent to site $i+1$.  The process is repeated until all chips are at distinct sites, and thus no further firing moves may be performed.  Example~\ref{ex1} below shows a complete firing sequence for $n=4$.

\begin{example}\label{ex1}
We begin with chips labeled 1 through 4 at the origin (the adjacent blank spaces represent the neighboring sites, which initially have no chips).

\begin{center}
\begin{tabular}{ccccc}

 & & 4 & & \\
 & & 3 & & \\
 & & 2 & & \\
\underline{\hspace{0.2 cm}} & \underline{\hspace{0.2 cm}} & \underline{1} & \underline{\hspace{0.2 cm}} & \underline{\hspace{0.2 cm}}

\end{tabular}
\end{center}

We can then choose to fire any pair of chips at the origin.  We choose chips 3 and 4, sending 3 to the left and 4 to the right.

\begin{center}
\begin{tabular}{ccccc}

 & & 2 & & \\
\underline{\hspace{0.2 cm}} & \underline{3} & \underline{1} & \underline{4} & \underline{\hspace{0.2 cm}}

\end{tabular}
\end{center}

We now fire chips 1 and 2:

\begin{center}
\begin{tabular}{ccccc}

 & 1 &  & 2 & \\
\underline{\hspace{0.2 cm}} & \underline{3} & \underline{\hspace{0.2 cm}} & \underline{4} & \underline{\hspace{0.2 cm}}

\end{tabular}
\end{center}

There are now two sites that can fire.  We choose to fire 2 and 4.

\begin{center}
\begin{tabular}{ccccc}

 & 1 & & & \\
\underline{\hspace{0.2 cm}} & \underline{3} & \underline{2} & \underline{\hspace{0.2 cm}} & \underline{4}

\end{tabular}
\end{center}

This leaves two more firing moves:

\begin{center}
\begin{tabular}{ccccc}

 & & 3 & & \\
\underline{1} & \underline{\hspace{0.2 cm}} & \underline{2} & \underline{\hspace{0.2 cm}} & \underline{4}

\end{tabular}
\end{center}

\begin{center}
\begin{tabular}{ccccc}

 & & & & \\
\underline{1} & \underline{2} & \underline{\hspace{0.2 cm}} & \underline{3} & \underline{4}

\end{tabular}
\end{center}

We have no more available firing moves, so we have reached a final configuration.
\end{example}

Note that in Example~\ref{ex1}, the chips are in sorted order from
left to right.  In fact, this system was shown in \cite{HMP16} to
always terminate in a unique final configuration, in which all chips
end in sorted order, as long as the number of chips $n$ is even.

We now proceed to our main result which is a new proof of the sorting property.  Let the total number of chips be $n=2m$ and label the chips $-m,-m+1,\cdots,-1,1,2,\cdots,m$.

\begin{theorem}[{\cite[Theorem 13]{HMP16}}] \label{configuration}
The labeled chip-firing process with $2m$ chips at the origin terminates at the final configuration with a single chip at every position from $-m$ to $-1$, and from $1$ to $m$.  Furthermore, the final position of each chip is equal to its label.
\end{theorem}

Given our choice of labels, showing that every chip ends up in the position designated by its label is equivalent to showing that the chips end up in sorted order. 

Define a partial order $P$ on the firing moves in the process.  Let

$$k_j=\text{ the }j^{th}\text{ to last firing move at site }k$$

The elements of the poset are ordered according to the relation: 

$$(k_1)_{j_1}\ge (k_2)_{j_2}\text{ if move }(k_1)_{j_1}\text{ must occur before move }(k_2)_{j_2}$$

$\;$\\

\begin{example}
Let $n=10$.  Figure~\ref{bottom} shows a bottom portion of the hasse diagram of $P$.  

\begin{center}
\begin{tikzpicture}\label{bottom}

\draw[black, thick] (0,0) -- (4,4);
\draw[black, thick] (1,-1) -- (5,3);
\draw[black, thick] (2,-2) -- (6,2);
\draw[black, thick] (3,-3) -- (7,1);
\draw[black, thick] (4,-4) -- (8,0);
\draw[black, thick] (0,0) -- (4,-4);
\draw[black, thick] (1,1) -- (5,-3);
\draw[black, thick] (2,2) -- (6,-2);
\draw[black, thick] (3,3) -- (7,-1);
\draw[black, thick] (4,4) -- (8,0);

\draw[red, thick] (6,2) -- (7,3);
\draw[red, thick] (5,3) -- (6,4);
\draw[red, thick] (4,4) -- (5,5);
\draw[red, thick] (4,4) -- (3,5);
\draw[red, thick] (3,3) -- (2,4);
\draw[red, thick] (2,2) -- (1,3);

\filldraw[black] (0,0) circle (2pt) node[right, gray] {$-4_1$} node[left] {$(0,4)$};
\filldraw[black] (1,-1) circle (2pt) node[right, gray] {$-3_1$} node[left] {$(0,3)$};
\filldraw[black] (1,1) circle (2pt) node[right, gray] {$-3_2$} node[left] {$(1,4)$};
\filldraw[black] (2,-2) circle (2pt) node[right, gray] {$-2_1$} node[left] {$(0,2)$};
\filldraw[black] (2,0) circle (2pt) node[right, gray] {$-2_2$} node[left] {$(1,3)$};
\filldraw[black] (2,2) circle (2pt) node[right, gray] {$-2_3$} node[left] {$(2,4)$};
\filldraw[black] (3,-3) circle (2pt) node[right, gray] {$-1_1$} node[left] {$(0,1)$};
\filldraw[black] (3,-1) circle (2pt) node[right, gray] {$-1_2$} node[left] {$(1,2)$};
\filldraw[black] (3,1) circle (2pt) node[right, gray] {$-1_3$} node[left] {$(2,3)$};
\filldraw[black] (3,3) circle (2pt) node[right, gray] {$-1_4$} node[left] {$(3,4)$};
\filldraw[black] (4,-4) circle (2pt) node[right, gray] {$0_1$} node[left] {$(0,0)$};
\filldraw[black] (4,-2) circle (2pt) node[right, gray] {$0_2$} node[left] {$(1,1)$};
\filldraw[black] (4,0) circle (2pt) node[right, gray] {$0_3$} node[left] {$(2,2)$};
\filldraw[black] (4,2) circle (2pt) node[right, gray] {$0_4$} node[left] {$(3,3)$};
\filldraw[black] (4,4) circle (2pt) node[right, gray] {$0_5$} node[left] {$(4,4)$};
\filldraw[black] (5,-3) circle (2pt) node[right, gray] {$1_1$} node[left] {$(1,0)$};
\filldraw[black] (5,-1) circle (2pt) node[right, gray] {$1_2$} node[left] {$(2,1)$};
\filldraw[black] (5,1) circle (2pt) node[right, gray] {$1_3$} node[left] {$(3,2)$};
\filldraw[black] (5,3) circle (2pt) node[right, gray] {$1_4$} node[left] {$(4,3)$};
\filldraw[black] (6,-2) circle (2pt) node[right, gray] {$2_1$} node[left] {$(2,0)$};
\filldraw[black] (6,0) circle (2pt) node[right, gray] {$2_2$} node[left] {$(3,1)$};
\filldraw[black] (6,2) circle (2pt) node[right, gray] {$2_3$} node[left] {$(4,2)$};
\filldraw[black] (7,-1) circle (2pt) node[right, gray] {$3_1$} node[left] {$(3,0)$};
\filldraw[black] (7,1) circle (2pt) node[right, gray] {$3_2$} node[left] {$(4,1)$};
\filldraw[black] (8,0) circle (2pt) node[right, gray] {$4_1$} node[left] {$(4,0)$};

\filldraw[fill=red,text=gray] (1,3) circle (2pt) node[right] {$-3_3$} node[left, black] {$(2,5)$};
\filldraw[fill=red,text=gray] (2,4) circle (2pt) node[right] {$-2_4$} node[left, black] {$(3,5)$};
\filldraw[fill=red,text=gray] (3,5) circle (2pt) node[right] {$-1_5$} node[left, black] {$(4,5)$};
\filldraw[fill=red,text=gray] (5,5) circle (2pt) node[right] {$1_5$} node[left, black] {$(5,4)$};
\filldraw[fill=red,text=gray] (6,4) circle (2pt) node[right] {$2_4$} node[left, black] {$(5,3)$};
\filldraw[fill=red,text=gray] (7,3) circle (2pt) node[right] {$3_3$} node[left, black] {$(5,2)$};

\end{tikzpicture}


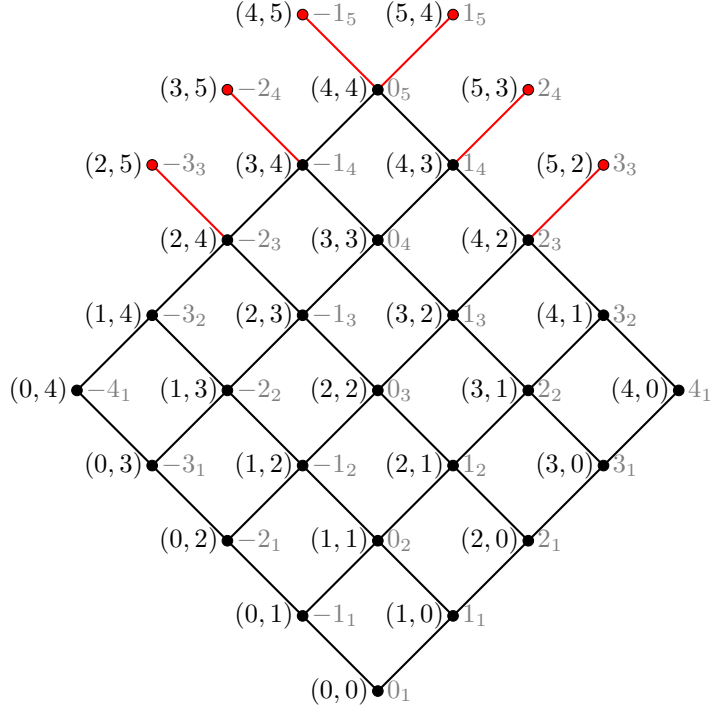
\captionof{figure}{Bottom of the firing move poset, $n=10$}

\end{center}

\end{example}

The proof of Theorem~\ref{configuration} consists of three main steps:

\begin{enumerate}

\item Prove the last moves in the process follow a locally confluent grid structure.

\item Bound the positions that chips can reach throughout the firing process.

\item Combine (1) and (2) to uniquely constrain a chip's final location.

\end{enumerate}

Our first goal is to show that the bottom portion of the Hasse diagram of $P$ always has the above grid structure.  For general $n$, the diagram will consist of an $m$ by $m$ black diamond.  At all but the two leftmost and two rightmost vertices at the top of the diamond, there is an additional edge oriented away from the diamond, drawn in red in the Figure.  The coordinates $(0,0)$ are assigned to the bottom vertex in the diamond.  A move of one unit up and to the right corresponds to an increase of 1 in the first coordinate, while a move of one unit up and to the left corresponds to an increase of 1 in the second coordinate.

Each column corresponds to a site, with the last move at that site appearing on the bottom border of the diamond, and with earlier moves moving up the diamond.  
We will often use ``firing move $(x,y)$'' or simply ``$(x,y)$'' to refer to the firing move that aligns with coordinates $(x,y)$ in the diagram.

 We  make repeated use of the following results from \cite{ALSSTW89}, see also \cite[Chapter 5]{Klivans18}:

\begin{theorem}\label{UnlabeledConfiguration}
The chip-firing process with $2m$ chips at the origin terminates at the final configuration with a single chip at every position from $-m$ to $-1$, and from $1$ to $m$.
\end{theorem}

\begin{theorem}\label{FiringNums}
Over the course of the chip-firing process with $2m$ chips at the origin, the number of firing moves at site $k$ is ${m-|k|+1 \choose 2}$ for $-m \le k \le m$.
\end{theorem}

The proof of the main theorem begins with the following lemma:

\begin{lemma}[Grid Structure]\label{diamond}

Let $(x,y)$ be a firing move such that $0 \le x,y \le m-1$.  Then the following conditions must hold:

1)  $(x,y) \leq (x+1, y)$ and $(x,y) \leq (x,y+1)$, i.e. 
the  firing move $(x,y)$ must take place after the moves $(x+1,y)$ and $(x,y+1)$, if such firing moves exist.

2) When the firing move $(x,y)$ occurs, there are exactly $2$ chips present at the site that is firing.

\end{lemma}

\begin{proof}

We first show that move $(0,0)$ (the last move at site 0) must take place after moves $(1,0)$ and $(0,1)$ (the last moves at sites $\pm 1$).  By Theorem~\ref{FiringNums}, site 0 fires $\frac{m(m+1)}{2}$ times in total, while sites 1 and -1 each fire $\frac{(m-1)m}{2}$ times in total.  Prior to the last fire at site $0$,  it has already lost $2\cdot (\frac{m(m+1)}{2}-1)$ chips due to firing.  In order to fire again, it needs to have 2 chips available, and since it started with $2m$ chips, this means that it must have gained at least $2+2\cdot (\frac{m(m+1)}{2}-1)-2m=(m-1)m$ chips from its neighbors.  However, this is equal to the total number of firing moves at sites 1 and -1, so all firing moves at those sites must occur prior to move $(0,0)$.  In addition, move $(0,0)$  must take place with only $2$ chips present.

\begin{center}
\begin{tikzpicture}[scale=0.9]

\draw[black, thick, dotted] (-5,-1) -- (-4,0);
\draw[black, thick, dotted] (-5,-1) -- (-6,0);

\filldraw[black] (-5,-1) circle (2pt) node[right] {$0_1$} node[left, gray] {$(0,0)$};
\filldraw[black] (-4,0) circle (2pt) node[right] {$1_1$} node[left, gray] {$(1,0)$};
\filldraw[black] (-6,0) circle (2pt) node[right] {$-1_1$} node[left, gray] {$(0,1)$};

\draw[black, thick, dotted] (0,-.5) -- (1,.5);
\draw[black, thick, dotted] (0,-.5) -- (-1,.5);
\draw[black, thick] (0,-.5) -- (-1,-1.5);

\filldraw[black] (0,-.5) circle (2pt) node[right] {$x_1$} node[left, gray] {$(x,0)$};
\filldraw[black] (1,.5) circle (2pt) node[right] {$(x+1)_1$} node[above, gray] {$(x+1,0)$};
\filldraw[black] (-1,.5) circle (2pt) node[right] {$(x-1)_2$} node[left, gray] {$(x,1)$};
\filldraw[black] (-1,-1.5) circle (2pt) node[right] {$(x-1)_1$} node [left, gray] {$(x-1,1)$};

\draw[black, thick] (0,-8) -- (4,-4);
\draw[black, thick] (0,-8) -- (-4,-4);
\draw[black, thick] (1,-7) -- (0,-6);
\draw[black, thick] (2,-6) -- (1,-5);
\draw[black, thick] (3,-5) -- (2,-4);
\draw[black, thick] (4,-4) -- (3,-3);
\draw[black, thick] (-1,-7) -- (0,-6);
\draw[black, thick] (-2,-6) -- (-1,-5);
\draw[black, thick] (-3,-5) -- (-2,-4);
\draw[black, thick] (-4,-4) -- (-3,-3);

\filldraw[black] (0,-8) circle (2pt) node[right] {$0_1$} node[left, gray] {$(0,0)$};
\filldraw[black] (1,-7) circle (2pt) node[right] {$1_1$} node[left, gray] {$(1,0)$};
\filldraw[black] (2,-6) circle (2pt) node[right] {$2_1$} node[left, gray] {$(2,0)$};
\filldraw[black] (3,-5) circle (2pt) node[right] {$3_1$} node[left, gray] {$(3,0)$};
\filldraw[black] (4,-4) circle (2pt) node[right] {$4_1$} node[left, gray] {$(4,0)$};
\filldraw[black] (-1,-7) circle (2pt) node[right] {$-1_1$} node[left, gray] {$(0,1)$};
\filldraw[black] (-2,-6) circle (2pt) node[right] {$-2_1$} node[left, gray] {$(0,2)$};
\filldraw[black] (-3,-5) circle (2pt) node[right] {$-3_1$} node[left, gray] {$(0,3)$};
\filldraw[black] (-4,-4) circle (2pt) node[right] {$-4_1$} node[left, gray] {$(0,4)$};
\filldraw[black] (0,-6) circle (2pt) node[right] {$0_2$} node[left, gray] {$(1,1)$};
\filldraw[black] (1,-5) circle (2pt) node[right] {$1_2$} node[left, gray] {$(2,1)$};
\filldraw[black] (2,-4) circle (2pt) node[right] {$2_2$} node[left, gray] {$(3,1)$};
\filldraw[black] (3,-3) circle (2pt) node[right] {$3_2$} node[left, gray] {$(4,1)$};
\filldraw[black] (-1,-5) circle (2pt) node[right] {$-1_2$} node[left, gray] {$(1,2)$};
\filldraw[black] (-2,-4) circle (2pt) node[right] {$-2_2$} node[left, gray] {$(1,3)$};
\filldraw[black] (-3,-3) circle (2pt) node[right] {$-3_2$} node[left, gray] {$(1,4)$};

\draw[black, thick, dotted] (5,0) -- (6,1);
\draw[black, thick, dotted] (5,0) -- (4,1);
\draw[black, thick] (5,0) -- (4,-1);
\draw[black, thick] (5,0) -- (6,-1);
\draw[black, thick] (4,-1) -- (5,-2);
\draw[black, thick] (6,-1) -- (5,-2);

\filldraw[black] (5,0) circle (2pt) node[left] {$(x,y)$};
\filldraw[black] (6,1) circle (2pt) node[left] {$(x+1,y)$};
\filldraw[black] (4,1) circle (2pt) node[left] {$(x,y+1)$};
\filldraw[black] (4,-1) circle (2pt) node [left] {$(x-1,y)$};
\filldraw[black] (6,-1) circle (2pt) node [left] {$(x,y-1)$};
\filldraw[black] (5,-2) circle (2pt) node [left] {$(x-1,y-1)$};

\end{tikzpicture}


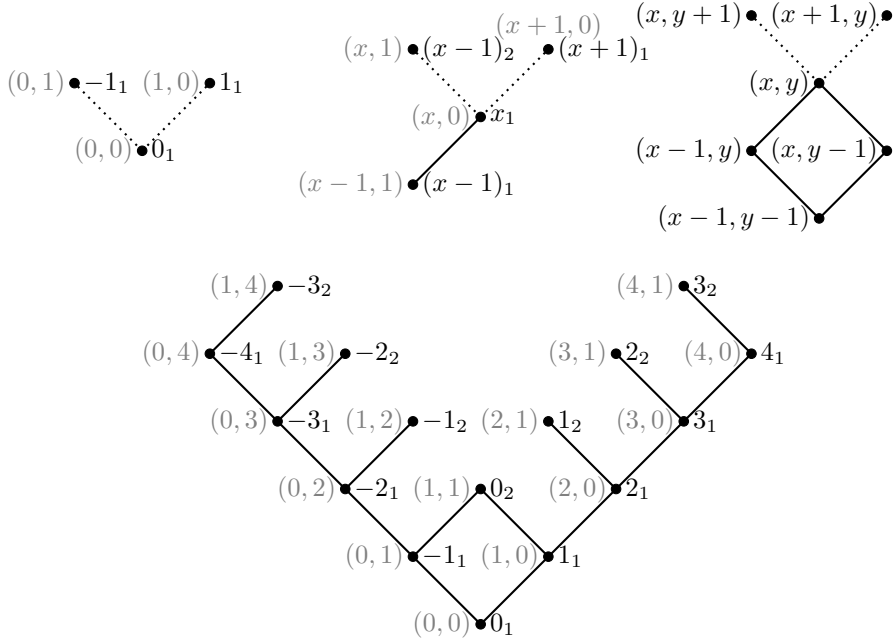
\captionof{figure}{Top left: The base case.  The last move at site 0 must take place after the last moves at sites 1 and -1.  Top middle: Inductive step for the last firing moves at each site. 
  Bottom: The edges that we've shown to exist after the completion of the above step for the case $n=10$.  Note the completed 1 by 1 diamond at the bottom, which we use for the main inductive step.  Top right: The main inductive step.  We assume the presence of the four edges directly below $(x,y)$ in the diagram and prove the existence of the two edges directly above those coordinates.}

\end{center}

Assume that the move $(0,x)$ (the last move at position $x$) must occur before move $(0,x-1)$ (the last move at $x-1$), for some $x>0$.  We want to show that $(0,x)$ takes place after $(0,x+1)$ (the last move at $x+1$, if it exists) and $(1,x)$ (the second to last move at $x-1$).  Site $x$ fires $\frac{(m-x)(m-x+1)}{2}$ times, so it has lost $(m-x)(m-x+1)-2$ chips due to firing before the last move occurs.  Its two neighbors fire a total of $\frac{(m-x+1)(m-x+2)}{2}+\frac{(m-x-1)(m-x)}{2}$ times, of which one is known to happen after the last move at site $x$.  As a result, site $x$ may have received at most $\frac{(m-x+1)(m-x+2)}{2}+\frac{(m-x-1)(m-x)}{2}-1=(m-x+1)(m-x)$ chips from its neighbors prior to its last firing move.  Thus, all firing moves at its neighbors need to occur before $x$ can fire for the final time with exactly 2 chips present.

We now move onto the main inductive step.  Induct on $x+y$, or equivalently, on the rows of Figure 1.  For fixed $k$ and all $x,y>0$ with $x+y=k$,  assume that move $(x,y)$ must take place before move $(x-1,y)$ and $(x,y-1)$, which both take place before move $(x-1,y-1)$.  Also assume that moves $(x-1,y)$, $(x,y-1)$, and $(x-1,y-1)$ all occur with 2 chips present.  This is illustrated in Figure 4.

Between the time that move $(x,y)$ is about to occur and the time that move $(x-1,y-1)$ is about to occur, site $x-y$ must lose 2 chips due to firing move $(x,y)$ and gain at least 2 chips due to moves $(x-1,y)$ and $(x,y-1)$.  Since there are 2 chips present when move $(x-1,y-1)$ occurs, this means that there are at most (and thus exactly) 2 chips present when move $(x,y)$ occurs.  Furthermore, no other moves may take place at neighboring sites $x-y-1$ and $x-y+1$ between the moves $(x,y)$ and $(x-1,y-1)$.  In particular, the moves $(x+1,y)$ and $(x,y+1)$ must take place before $(x,y)$ in order for this to be possible.
\end{proof}

\begin{remark}

It is not immediately clear why this same argument cannot be extended to show that all firing moves in the entire chip firing process must follow this grid structure.  The reason is that our inductive step breaks down at the left and right corners of the diamond.  There are no firing moves at sites $\pm m$, and there is only one each at $\pm (m-1)$.

We see that because there is one fewer firing move that must occur directly after, say, the move $(3,m)$, this gives more freedom to what can happen at that firing move.  It allows the move to take place either with 3 chips present, or before all other firing moves at neighboring sites have occurred.  This effect cascades upward, and ultimately means that all other firing moves in the process (other than the first two at site 0), can have a wider range of numbers of chips present, preventing this structure from applying elsewhere.  For the rest of the Hasse diagram for $n=10$, as well as for the odd $n=11$, see Figure \ref{fign10}.

\end{remark}

Next, we consider bounds on the locations of chips with given labels.  The following is a weaker bound than the one provided in \cite{HMP16}.  Having proved the existence of the grid structure, we do not need as tight bounds as in \cite{HMP16}.

\begin{lemma}[Position Bounds]\label{bounds}

The position of chip $k$ $(k<0)$ must never exceed $k+m$ at any point in the chip firing process.  Similarly, the position of chip $k$ $(k>0)$ must never be less than $k-m$ at any point in the chip firing process.

\end{lemma}

\begin{proof}

We proceed by strong induction on $k<0$.  $k=-m$ is the smallest label that a chip can have.  Thus, at no point may the position of this chip increase from its initial position of 0.  As a result, the position of $-m$ may never exceed $-m+m=0$.

Suppose that for all chips with labels less than $k$ ($k<-1$), we have that the positions of these chips may not exceed $k+m$.  Consider the chip with label $k+1$.  If this chip ever reaches position $k+1+m$, then by our inductive assumption, it must be the smallest chip to reach this position.  As a result, it cannot be fired to the right and increase its position from $k+1+m$.  Thus, the position of chip $k+1$ can never exceed $k+1+m$, and the result for negative $k$ follows by induction.  The result for positive $k$ is analogous.

\end{proof}

  Visually, the next Lemma  proves that for $k<0$, chip $k$ must always stay at or to the left of the line $y=-k-1$ in the grid, while for $k>0$, chip $k$ must stay at or to the right of the line $x=k-1$, see Figure~\ref{lines}.

\begin{lemma}\label{mainlemma}

For each chip $k$ with $-m \le k <0$, and each $x$ with $0 \le x \le m-1$, the position of chip $k$ may not exceed $x+k+1$ immediately preceding firing move $(x,-k-1)$.  Similarly, for each $0 < k \le m$ and $0 \le y \le m-1$, the position of chip $k$ must be at least $k-1-y$ immediately preceding firing move $(k-1,y)$.

\end{lemma}

\begin{proof}

  We proceed by induction, first on coordinate $x$, and then on chip $k$.  For the base case, let $k=-m$ and $x=m-1$; $x$ will decrease to 0, while $k$ increases to -1.

We must first show that chip $-m$ has position at most 0 prior to firing move $(m-1,m-1)$.  But by Lemma~\ref{bounds}, the position of chip $-m$ may never exceed 0.

Next, fix $k=-m$ and induct on $x$.  Suppose that chip $-m$ has position at most $x-m+1$ prior to firing move $(x,m-1)$.  This leaves us with two cases: either the position of chip $-m$ is at most $x-m$ prior to move $(x,m-1)$, in which case its position will not change as a result of move $(x,m-1)$, or its position is exactly equal to $x-m+1$ when move $(x,m-1)$ occurs.  From Lemma~\ref{diamond},there are only 2 chips present when move $(x,m-1)$ occurs, so chip $-m$ must be fired by this move.  Because $-m$ is the smallest chip, it must be fired to the left, to position $x-m$.  Since no moves may occur at site $x-m$ between moves $(x,m-1)$ and $(x-1,m-1)$, chip $-m$ must still be at site $x-m$ prior to firing move $(x-1,m-1)$.

Now, induct on $k$.  Assume that for all $k'\le k$ and $0 \le x \le m-1$, the position of chip $k'$ may not exceed $x+k'+1$ immediately preceding firing move $(x,-k'-1)$.  Fixing chip $k+1$ and setting $x=m-1$, we need the position of chip $k+1$ to not exceed $k+1+m$ prior to firing move $(m-1,-k-2)$.  Again by Lemma~\ref{bounds}, chip $k+1$ may never exceed that position.

Finally, assume that chip $k+1$ has position at most $x+k+2$ prior to firing move $(x,-k-2)$.  We again have two cases: either the position of chip $k+1$ is at most $x+k+1$ prior to move $(x,-k-2)$, in which case its position will not change as a result of move $(x,-k-2)$, or its position is exactly equal to $x+k+2$ when move $(x,-k-2)$ occurs.  From Lemma~\ref{diamond}, there are only 2 chips present when move $(x,-k-2)$ occurs, so chip $k+1$ must be fired by this move.  Since move $(x,j)$ has already taken place for all $j<k-2$, all chips with values less than or equal to $k$ must have position at most $x+k$.  As a result, chip $k+1$ must be the smallest chip at site $x+k+2$ when move $(x,-k-2)$ occurs, so it is fired to the left, to position $x+k+1$.  Therefore chip $k+1$ has position at most $x+k+1$ after move $(x,-k-2)$, and since no moves may take place at site $x+k+1$ between moves $(x,-k-2)$ and $(x-1,-k-2)$, chip $k+1$ must still have position at most $x+k+1$ immediately prior to move $(x-1,-k-2)$ as desired.  The positive case follows similarly.

\end{proof}

\begin{center}
\begin{tikzpicture} 

\draw[black, thick] (0,0) -- (4,4);
\draw[black, thick] (1,-1) -- (5,3);
\draw[black, thick] (2,-2) -- (6,2);
\draw[black, thick] (3,-3) -- (7,1);
\draw[black, thick] (4,-4) -- (8,0);
\draw[black, thick] (0,0) -- (4,-4);
\draw[black, thick] (1,1) -- (5,-3);
\draw[black, thick] (2,2) -- (6,-2);
\draw[black, thick] (3,3) -- (7,-1);
\draw[black, thick] (4,4) -- (8,0);
\draw[red, thick, ->] (1,3) -- (5.9,-1.9);
\draw[red, thick, ->] (1.5,2.5) --(2,2.5);
\draw[red, thick, ->] (6,4) -- (1.1,-0.9);
\draw[red, thick, ->] (5.5,3.5) -- (5,3.5);

\filldraw[black] (0,0) circle (2pt) node[right] { $\text{ }1$} node[left] {$-5\text{ }$};
\filldraw[black] (1,1) circle (2pt) node[right] { $\text{ }2$} node[left] {$-5\text{ }$};
\filldraw[black] (1,-1) circle (2pt) node[right] { $\text{ }1$} node[left] {$-4\text{ }$};
\filldraw[black] (2,2) circle (2pt) node[right] { $\text{ }3$} node[left] {$-5\text{ }$};
\filldraw[black] (2,0) circle (2pt) node[right] { $\text{ }2$} node[left] {$-4\text{ }$};
\filldraw[black] (2,-2) circle (2pt) node[right] { $\text{ }1$} node[left] {$-3\text{ }$};
\filldraw[black] (3,3) circle (2pt) node[right] { $\text{ }4$} node[left] {$-5\text{ }$};
\filldraw[black] (3,1) circle (2pt) node[right] { $\text{ }3$} node[left] {$-4\text{ }$};
\filldraw[black] (3,-1) circle (2pt) node[right] { $\text{ }2$} node[left] {$-3\text{ }$};
\filldraw[black] (3,-3) circle (2pt) node[right] { $\text{ }1$} node[left] {$-2\text{ }$};
\filldraw[black] (4,4) circle (2pt) node[right] { $\text{ }5$} node[left] {$-5\text{ }$};
\filldraw[black] (4,2) circle (2pt) node[right] { $\text{ }4$} node[left] {$-4\text{ }$};
\filldraw[black] (4,0) circle (2pt) node[right] { $\text{ }3$} node[left] {$-3\text{ }$};
\filldraw[black] (4,-2) circle (2pt) node[right] { $\text{ }2$} node[left] {$-2\text{ }$};
\filldraw[black] (4,-4) circle (2pt) node[right] { $\text{ }1$} node[left] {$-1\text{ }$};
\filldraw[black] (5,3) circle (2pt) node[right] { $\text{ }5$} node[left] {$-4\text{ }$};
\filldraw[black] (5,1) circle (2pt) node[right] { $\text{ }4$} node[left] {$-3\text{ }$};
\filldraw[black] (5,-1) circle (2pt) node[right] { $\text{ }3$} node[left] {$-2\text{ }$};
\filldraw[black] (5,-3) circle (2pt) node[right] { $\text{ }2$} node[left] {$-1\text{ }$};
\filldraw[black] (6,2) circle (2pt) node[right] { $\text{ }5$} node[left] {$-3\text{ }$};
\filldraw[black] (6,0) circle (2pt) node[right] { $\text{ }4$} node[left] {$-2\text{ }$};
\filldraw[black] (6,-2) circle (2pt) node[right] { $\text{ }3$} node[left] {$-1\text{ }$};
\filldraw[black] (7,1) circle (2pt) node[right] { $\text{ }5$} node[left] {$-2\text{ }$};
\filldraw[black] (7,-1) circle (2pt) node[right] { $\text{ }4$} node[left] {$-1\text{ }$};
\filldraw[black] (8,0) circle (2pt) node[right] { $\text{ }5$} node[left] {$-1\text{ }$};
\filldraw[red] (1,3) circle (0.1pt) node[right] { $\text{ }3$};
\filldraw[red] (6,4) circle (0.1pt) node[left] {$-4\text{ }$};

\node at (4,5.5) {Sites};
\node at (-1,5) {$-5$};
\node at (0,5) {$-4$};
\node at (1,5) {$-3$};
\node at (2,5) {$-2$};
\node at (3,5) {$-1$};
\node at (4,5) {$0$};
\node at (5,5) {$1$};
\node at (6,5) {$2$};
\node at (7,5) {$3$};
\node at (8,5) {$4$};
\node at (9,5) {$5$};

\node[left] at (-1,-5) {$-5$};
\node[left] at (0,-5) {$-4$};
\node[left] at (1,-5) {$-3$};
\node[left] at (2,-5) {$-2$};
\node[left] at (3,-5) {$-1$};
\node[right] at (5,-5) {$1$};
\node[right] at (6,-5) {$2$};
\node[right] at (7,-5) {$3$};
\node[right] at (8,-5) {$4$};
\node[right] at (9,-5) {$5$};

\draw[->, thick] (0,0) -- (-1,-5);
\draw[->, thick] (1,-1) -- (0,-5);
\draw[->, thick] (2,-2) -- (1,-5);
\draw[->, thick] (3,-3) -- (2,-5);
\draw[->, thick] (4,-4) -- (3,-5);
\draw[->, thick] (4,-4) -- (5,-5);
\draw[->, thick] (5,-3) -- (6,-5);
\draw[->, thick] (6,-2) -- (7,-5);
\draw[->, thick] (7,-1) -- (8,-5);
\draw[->, thick] (8,0) -- (9,-5);

\end{tikzpicture}


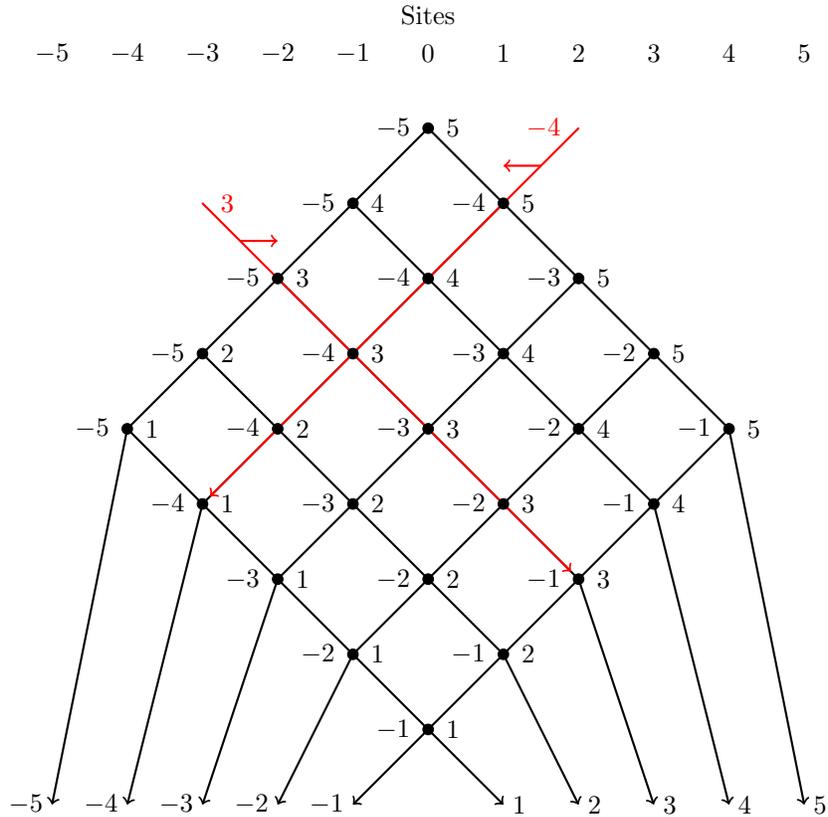
\captionof{figure}{Diagram for Lemma~\ref{mainlemma} and Lemma~\ref{mainthm}.  For each firing move shown, the chip written to the left of the vertex must be at or to the left of that site when the firing move occurs, while the chip written to the right must be at or to the right of that site.  The arrows at the bottom show how the last firing moves at each site send each of the chips to their final, sorted positions.  As examples, the red arrows shown are the respective left and right bounds for the positions of chips $3$ and $-4$ when given firing moves occur.} \label{lines}

\end{center}

We are now ready for our main result.

\begin{theorem}\label{mainthm}

When the labeled chip-firing process terminates, each chip $k$ is at position $k$.

\end{theorem}

\begin{proof}

By Lemma~\ref{mainlemma}, chip $k$ ($k<0$) must have position less than or equal to $k+1$ immediately prior to firing move $(0,-k-1)$.  Using the same argument as in Lemma~\ref{mainlemma}, it must have position at most $k$ immediately after that firing move.  Since no firing moves may occur at positions less than or equal to $k$ after firing move $(0,-k-1)$, the final position of chip $k$ must be less than or equal to $k$.  The only way to satisfy this condition simultaneously for all $k<0$ is if each chip $k$ is at position $k$.  The case for $k>0$ follows similarly.

\end{proof}

\section{Related Results}\label{related}

Using the methods above, we are able to prove confluence for a number of similar cases. All of the results essentially come down to the same principles:
\begin{itemize}
\item There is a diamond of moves at the end of the process that satisfies local confluence

\item The initial portion of the chip-firing process permutes the chips by only a ``small amount''

\item The final portion of the chip-firing process can always ``fix'' the resulting errors. 
\end{itemize}
In general, larger diamonds lead to more straightforward proofs, while smaller diamonds require stricter bounds.

\subsection{Alternative Initial Configuration}

Consider an initial configuration with $2n+1$ chips: chips $-n$ through $-1$ are at site $-1$, and chips $1$ through $n+1$ are at the origin.  We have the following result, originally  conjectured in \cite{GHMP2}.

\begin{theorem}
In the labeled chip-firing process beginning with chips $-n$ through $-1$ at site $-1$ and chips $1$ through $n+1$ at the origin, each chip $k$ ends at site $k$.
\end{theorem}

We omit the details here, but the proof is virtually identical to that of \ref{mainthm}.  The diamond becomes an $n$ by $n+1$ rectangle, so the $n+1$ positive chips have position lower bounds, and the $n$ negative chips have position upper bounds.  An argument similar to Lemma \ref{mainlemma} can be used to track the chips through this modified diamond until they reach their final positions.

\subsection{Multiple Edges}

Consider the 1 dimensional grid in which each edge is replaced with $r$ edges.  A firing move sends $r$ chips to the left and $r$ chips to the right.  In the labeled setting, each firing move consists of choosing $2r$ chips at the same site and  sending the $r$ chips with the smallest labels to the left, and the $r$ chips with the largest labels to the right.

Starting with a multiple of $2r$ chips, the final configuration has $r$ chips at each nonempty site, so the notion of sorting does not directly apply.  The next Theorem (previously Conjecture 24 in \cite{HMP16}) shows the strongest possible notion of sorting does hold. 

\begin{theorem}
Consider the labeled chip-firing process on the line with $r$ copies of each edge and $2rm$ chips initially at the origin.  In the final configuration, for all chips $a$ and $b$ with $a<b$, the final position of $a$ is at or to the left of the final position of $b$.  
\end{theorem}

\begin{proof}

The proof is similar to that of Theorem~\ref{mainthm}.    Suppose that we start with an initial unlabeled configuration with $2rm$ chips and then fire to completion.  If, at each step, we divide the number of chips at each site by $r$, then the resulting process begins with $2m$ chips, and each firing move sends one to the left and one to the right.  Since this can be reversed by multiplying all chip configurations in the original process by $r$, there is a bijection between chip configurations in the $r$ edge case and the 1 edge case, with the same available firing moves, and thus the same partial order.

  Using the same argument as in Lemma~\ref{bounds}, we get that the $r$ smallest chips cannot move past position 0, the next $r$ smallest chips cannot move past position $1$, and so on, simply because the $r$ smallest chips at any given site cannot be fired to the right from that site.  Then, using the arguments from Lemma~\ref{mainlemma} and Theorem~\ref{mainthm}, we can show that those groups of $r$ chips satisfy the same position bounds in the diamond as did the original single chips.   Thus the $r$ smallest chips end up at position $-m$, the next $r$ smallest chips end at position $-m+1$, and so on, giving the analogous results in the $r$ edge case.

\end{proof}

\subsection{Self-Loops at the Origin}

Next consider the $1$ dimensional grid with self-loops at the origin.  Firing moves at all sites but the origin are the same as in the original problem.  If there are $s$ self-loops at the origin, then a firing move at the origin consists of choosing $s+2$ chips at the origin and firing the smallest to the left and the largest to the right.  Here, we reprove Theorem 20 from \cite{HMP16}:

\begin{theorem}
Consider the labeled chip-firing problem with $s$ self-loops at the origin.  For $n$ chips with $n\ge s$ and $n\equiv s$ mod $2$, the final configuration is weakly sorted with $s$ chips at the origin and one chip at every other site from $\frac{s-n}{2}$ to $\frac{n-s}{2}$.
\end{theorem}

\begin{proof}
As in the multiple edge case, there is a bijection between unlabeled states in this problem and in the original problem.  All unlabeled states in this chip-firing process must have at least $s$ chips at the origin.  Suppose that we start with an initial unlabeled configuration with $n$ chips.  If, at each step, we subtract $s$ from the number of the chips at the origin, then the resulting process begins with $n-s$ chips, and each firing move sends one to the left and one to the right.  This can be reversed by adding $s$ chips at the origin at each step of the original chip-firing process.  We again get the same available firing moves, and since $n-s$ must be even, we get the same partial order.

The same bounds on chip positions exist as before, which in turn give the same bounds on chip positions that occur at various firing moves in the diamond.  While these bounds are only ever applied to the $m$ smallest chips and the $m$ largest chips, the remaining chips are eventually forced to be at the origin by the fact that specific other chips are required to occupy every other site.  This results in each of the $m$ smallest and $m$ largest chips being placed in sorted order at every position but the origin, which results in all of the chips being weakly sorted.

\end{proof}

\subsection{Exponentially Many Edges}

We now consider a new problem with more structure in the poset than the original problem.  This results in certain parts of the proof becoming simpler, as the chips' movements are more constrained.

We now vary the number of edges between each pair of adjacent nodes.  We choose a $t$, and then for every $k$ from 0 to $t$, we place $2^{t-k}$ edges between nodes $k$ and $k+1$, as well as between nodes $-k$ and $-k-1$.  Past nodes $t+1$ and $-t-1$, all other pairs of adjacent nodes are connected by a single edge.  If a site has $a$ leftward edges and $b$ rightward edges, we can fire at that site by choosing $a+b$ chips, and then sending the smallest $a$ to the left and the largest $b$ to the right.  We will show that a process beginning with $2^{t+2}$ chips must end in a weakly sorted configuration.  We first establish a relationship between the numbers of firing moves at each site:

\begin{lemma}
For $0\le k \le t$, site $k$ fires exactly 2 more times than site $k+1$.  Similarly, for $-t \le k \le 0$, site $k$ fires exactly 2 more times than site $k-1$.
\end{lemma}

\begin{proof}
We will begin with the $k\ge 0$ case.  The negative case follows by symmetry.

We proceed by induction.  We define $f(k)$ to be the number of firing moves at site $k$.  By counting the number of chips lost or gained due to firing, the number of chips at site 0 at the end of the process is equal to $2^{t+2}+2^t(f(1)+f(-1)-2f(0))$.  By symmetry, this is equal to $2^t(2f(1)-2f(0))$  The number of chips can't be negative, and if there are at least $2^{t+1}$ chips left at the end, site 0 can fire.  As a result, we must have $f(0)=2+f(1)=2+f(-1)$ in order to have a valid final configuration.  This results in a final configuration with 0 chips at the origin.

Now, we suppose that $f(k-1)=2+f(k)$ for some $k$ from 1 to $t-1$.  We have that the final number of chips at $k$ is equal to $2^{t-|k|+1}f(k-1)-(2^{t-|k|+1}+2^{t-|k|})f(k)+2^{t-|k|}f(k+1)$.  By our inductive assumption, we get that $f(k-1)=f(k)+2$, so this simplifies to $2^{t-|k|+2}-2^{t-|k|}f(k)+2^{t-|k|}f(k+1)$.  Again, in order to have a stable final configuration, we need $f(k)=2+f(k+1)$, with $2^{t-|k|+1}$ chips at site $k$ in the final configuration.  The result follows by induction.
\end{proof}

In addition to obtaining properties about the numbers of firing moves, we also obtain almost the entire final configuration of the process.  There must be 1 chip each at sites greater than $t+1$ and less than $-t-1$.  Since 1 chip is not enough to fire at any site, this means that these chips must occupy sites $t+2$ and $-t-2$ in the final configuration.  Since no firing moves occur at these sites, site $t+1$ must fire exactly once, meaning that for each $-t-1 \le k \le t+1$, site $k$ must fire $2(t-|k|)+3$ times.

Next we prove a result analogous to showing the existence of the diamond in Section 2.  In this case, the grid structure actually extends to the entire firing order poset, rather than just a small collection of moves at the end, see Figure~\ref{figexp}.

\begin{lemma}
For $1 \le k\le t+1$ and $1 \le j \le 2(t-|k|)+3$, the $j^{th}$ move at site $k$ must occur between moves $j+1$ and $j+2$ at site $k-1$.  Similarly, for $-t-1 \le k \le -1$ and $1 \le j \le 2(t-|k|)+3$, the $j^{th}$ move at site $k$ must occur between moves $j+1$ and $j+2$ at site $k+1$.  Furthermore, all firing moves in the entire process, with the exceptions of the first firing move at each site from sites $-t$ to $t$, leave 0 chips behind at their respective sites after firing.
\end{lemma}

The proof is virtually identical to the proof of Lemma \ref{diamond}, but the different numbers of chips being fired in each direction enable the grid structure to extend throughout the rest of the poset.  If the second move at site $k$ is known to occur before the first move at site $k+1$ and the third move at site $k-1$, then it must take place with exactly $3(2^{t-k+1})-2^{t-k+1}-2^{t-k}=2^{t-k+1}+2^{t-k}$ chips present.  This is exactly the number needed for the site to fire, so it leaves 0 chips at that site after firing, and the move must take place after the second move at site $k-1$.  The rest of the proof uses the same induction argument as Lemma\ref{diamond}.

We then obtain the following result.  The smallest and largest chips must repeatedly return to within one site of their desired final position

\begin{lemma}\label{expbounds}
For $0 \le k \le t$, and $2 \le j \le 2(t-k)+3$, the following holds.  After the $j^{th}$ move at site $k$, the largest $2^{t-k}$ chips are at position greater than $k$.  Similarly, for $-t \le k \le 0$ and $2 \le j \le 2(t-k)+3$, we have that after the $j^{th}$ move at site $k$, the smallest $2^{t+k}$ chips are at position less than $k$.
\end{lemma}

\begin{proof}
We will prove the $k\ge 0$ case, and the negative case follows similarly.  We induct on $k$.  For $k=0$, we have that site 0 fires $2^t$ chips to the right, so the largest $2^t$ chips can never be to the left of site 0.  Since all firing moves starting with the second must fire all of their chips, this means that the largest $2^t$ chips must either already be at a site $k>0$ when such a move happens, or they must be fired to the right by any such move.

Now, we suppose that this result holds for position $k-1$.  By our inductive assumption, the largest $2^{t-k+1}$ chips must be at position at least $k$ immediately after the $j^{th}$ firing move at $k-1$, for any $j\ge 2$.  Since chips can only go back to site $k-1$ through a firing move at site $k$, and since only one such move may occur between any two consecutive moves at site $k-1$, this means that whenever a move occurs at site $k$, the largest $2^{t-k+1}$ chips must be at position at least $k$.  Since any such move fires $2^{t-k}$ chips to the right, the largest $2^{t-k}$ chips must be fired to the right by this move if they aren't already at positions greater than $k$.  This completes the induction.
\end{proof}

This gives us enough information to complete our result.

\begin{theorem}
In the exponential edge problem beginning with $2^{t+2}$ chips at the origin, the final configuration has all chips in weakly sorted order.
\end{theorem}

\begin{proof}
We prove this by considering the last time that each chip is fired, beginning at position $t+1$ and inducting downward toward 0.  We will again show the positive case, with the negative case following by symmetry.

By Lemma \ref{expbounds}, the largest chip must be at position $t+1$ after the second firing move at position $t$.  Thus, it must be included in the lone firing move at position $t+1$, and since it is the largest chip, it must be fired to the right.  This places the largest chip at a final position of $t+2$.

Now, we assume that the largest $2^{t-k}$ chips reach their correct final positions in order to be weakly sorted.  We have that after the second to last move at site $k-1$, the largest $2^{t-k+1}$ chips must be at position at least $k$.  The largest $2^{t-k}$ chips are assumed to be in their correct final positions when the last move at site $k$ occurs, which means that the next $2^{t-k}$ chips must be the largest $2^{t-k}$ chips at site $k$.  They are thus fired to the right by the firing move at site $k$, reaching the desired final position at site $k+1$.

\end{proof}

The larger grid structure helps to simplify the proof from that of Theorem~\ref{mainthm}.  Instead of tracking chip positions throughout the process and then bounding their positions throughout the diamond, we can show that chips tend to be near where they're supposed to be for almost the entire process, finally slotting into their desired positions at the end.  The Hasse diagram for the firing move poset appears in Figure 7.

\subsection{One Self-Loop at Every Site}

We now turn to a case that is further from the original problem.  Consider the  $1$ dimensional grid graph with a self-loop at every vertex.  Every firing move requires 3 chips in order to fire, with the smallest of the three moving to the left, and the largest of the three moving to the right.

The goal of this section is to prove a previously unproved conjecture from \cite{HMP16}: for an initial configuration on this graph with $n \equiv 3$ mod $4$ chips at the origin, the final configuration is weakly sorted.  We first prove some analogous results in this setting:

\begin{lemma}\label{config}
The chip-firing process on the self-loop graph, beginning with $4m-1$ chips at the origin, terminates with 1 chip each at positions $-m$, 0, and $m$; 2 chips each at positions $-m+1$ through $-1$ and $1$ through $m-1$; and 0 chips elsewhere.
\end{lemma}

\begin{proof}

We modify the proof of Theorem 2 from \cite{ALSSTW89}.  Since there are finitely many chips on an infinite, connected graph, the process must terminate in a unique final configuration.  Thus, it suffices to show a single sequence of firing moves that leads to this configuration.  We will deal with a sequence that leads to the following intermediate configurations:

\begin{center}
\begin{tabular}{ccccccc}

    &    &    & $4n-1$ &    &    &    \\
    &    & 1 & $4n-3$ & 1 &    &    \\
    &    & 2 & $4n-5$ & 2 &    &    \\
    & 1 & 2 & $4n-7$ & 2 & 1 &    \\
    & 2 & 2 & $4n-9$ & 2 & 2 &    \\
 1 & 2 & 2 & $4n-11$ & 2 & 2 & 1 \\
 2 & 2 & 2 & $4n-13$ & 2 & 2 & 2 \\
    &    &    & $\vdots$ &    &    &   
    
\end{tabular}
\end{center}

In each step, we remove 2 chips from the origin and place 1 chip at each the two closest sites to the origin that do not yet have 2 chips.  Each step only requires the use of 3 chips at the origin, so we treat each case as if only 3 chips are present there initially.  In case 1 (2 chips initially at the farthest sites), we make all possible firing moves simultaneously and obtain the following pattern:

\begin{center}
\begin{tabular}{ccccccccccc}

   & 2 & 2 & 2 & 2 & 3 & 2 & 2 & 2 & 2 &   \\
   & 2 & 2 & 2 & 3 & 1 & 3 & 2 & 2 & 2 &   \\
   & 2 & 2 & 3 & 1 & 3 & 1 & 3 & 2 & 2 &   \\
   & 2 & 3 & 1 & 3 & 1 & 3 & 1 & 3 & 2 &   \\
   & 3 & 1 & 3 & 1 & 3 & 1 & 3 & 1 & 3 &   \\
1 & 1 & 3 & 1 & 3 & 1 & 3 & 1 & 3 & 1 & 1\\
1 & 2 & 1 & 3 & 1 & 3 & 1 & 3 & 1 & 2 & 1\\
1 & 2 & 2 & 1 & 3 & 1 & 3 & 1 & 2 & 2 & 1\\
1 & 2 & 2 & 2 & 1 & 3 & 1 & 2 & 2 & 2 & 1\\
1 & 2 & 2 & 2 & 2 & 1 & 2 & 2 & 2 & 2 & 1\\
\end{tabular}
\end{center}

And in case 2, with 1 chip at the farthest sites initially:

\begin{center}
\begin{tabular}{ccccccccccc}

1 & 2 & 2 & 2 & 2 & 3 & 2 & 2 & 2 & 2 & 1\\
1 & 2 & 2 & 2 & 3 & 1 & 3 & 2 & 2 & 2 & 1\\
1 & 2 & 2 & 3 & 1 & 3 & 1 & 3 & 2 & 2 & 1\\
1 & 2 & 3 & 1 & 3 & 1 & 3 & 1 & 3 & 2 & 1\\
1 & 3 & 1 & 3 & 1 & 3 & 1 & 3 & 1 & 3 & 1\\
2 & 1 & 3 & 1 & 3 & 1 & 3 & 1 & 3 & 1 & 2\\
2 & 2 & 1 & 3 & 1 & 3 & 1 & 3 & 1 & 2 & 2\\
2 & 2 & 2 & 1 & 3 & 1 & 3 & 1 & 2 & 2 & 2\\
2 & 2 & 2 & 2 & 1 & 3 & 1 & 2 & 2 & 2 & 2\\
2 & 2 & 2 & 2 & 2 & 1 & 2 & 2 & 2 & 2 & 2\\
\end{tabular}
\end{center}

In both cases, there is a line of alternating 3's and 1's that expands until extra chips are deposited at the farthest points, and then contracts again.  Starting with $4m-1$ chips and then alternating between configurations of the two above forms until no more firing moves may be performed yields a configuration of the desired form.

\end{proof}

\begin{lemma}
In the chip-firing process on the self-loop graph, beginning with $4m-1$ chips at the origin, and running to completion, the number of firing moves at site $k$ for $-m \le k \le m$ is equal to $(m-|k|)^2$.
\end{lemma}

\begin{proof}
We will prove the results for nonnegative $k$, and negative $k$ will follow by symmetry.  Since the final configuration has no chips past position $m$, there must be no firing moves at position $m$ throughout the process.  For positions $1 \le k \le m-1$, the number of chips sent right by site $k$ must be equal to the number of chips sent to the left by $k+1$, plus the number of chips remaining at sites greater than $k$.  Let $f(k)$ be the number of firing moves at site $k$.  This gives the recurrence

$$f(k)=f(k+1)+2(m-k)-1\text{ for } 1\le k \le m-1$$

A similar argument gives the following recurrence relation for $k=0$:

$$f(0)=\frac{1}{2}(f(1)+f(-1)+4m-2)$$

If we define the first recurrence analogously for negative $k$, these recurrences have a unique solution with each site $k$ firing $(m-|k|)^2$ times.

\end{proof}

These lemmas allow us to prove the existence of an identical diamond to the one from Lemma \ref{diamond}.  As the proof of its existence is virtually identical to the proof of Lemma \ref{diamond}, we omit it here.

Now, Lemma \ref{config} suggests a new chip-numbering scheme for this problem.  If we perform this chip-firing process beginning with $4m-1$ chips at the origin, we label the chips as follows: we assign the labels $-m$, 0, and $m$ to one chip each, and we assign the labels $1-m$ through $-1$ and $1$ through $m-1$ to two chips each.  If we perform a firing move involving two chips with the same label, we arbitrarily treat one of them as the smaller chip.  The weak sorting result is now similar to Theorem \ref{mainthm}; we prove that the final position of each chip is equal to its label.

We first need bounds analogous to Lemma 12 from \cite{HMP16}.  We omit the proof here, it is similar to the proof of  Lemma 12.

\begin{lemma}\label{loopBounds}
Throughout the chip firing process, the position of a chip numbered $k$ $(k>0)$ must always be between $\lfloor \frac{k-m}{2} \rfloor$ and  $\lfloor \frac{k+m}{2} \rfloor$ inclusive.  The position of a chip numbered $k$ $(k<0)$ must always be between $\lceil \frac{k-m}{2} \rceil$ and  $\lceil \frac{k+m}{2} \rceil$ inclusive.  Chip 0 must remain between position $\lceil \frac{-m}{2} \rceil$ and  $\lfloor \frac{m}{2} \rfloor$ inclusive.  Furthermore, if $k+m$ is even (for the lower bounds) or $m-k$ is even (for the upper bounds), at most one chip with a given label may satisfy the equality cases of the above bounds at any given time.
\end{lemma}

The proof of the corresponding lemma in \cite{HMP16} works by comparing the  positions of all chips greater than or equal to $k$ to the final positions that those chips would have occupied if all smaller chips had been removed from the process at the beginning.  The conclusion that the leftmost possible position for chip $k$ cannot be further left that the farthest left position in the other process still applies here.

We then want to show how many chips with labels in certain ranges can appear in certain regions.  Such bounds will often result from showing that a violation of these conditions would force certain chips to violate their bounds from Lemma \ref{loopBounds}.

Given an instance of the labeled chip-firing process $p$ and a chip $c$, define $f_p(c)$ to be the first of the diamond firing moves of $p$ for which chip $c$ is present.  Then define $s_p(c)$ to be the position at which that firing move occurs.  Note that $s_p$ maps each chip to a site, and thus defines a chip configuration.  We refer to this as the \textit{diamond configuration} of the process $p$.

We then have the following lemma:

\begin{lemma}\label{lemma1}
For any $k$ with $-m-1 \le k \le 0$, and any $l$ with $0 \le l \le k+m-1$, there must be at most $k+m-l-1$ chips $c$ such that the value of chip $c$ is less than $k$ and such that $s_p(c)>l$.
\end{lemma}

\begin{proof}
We will proceed by contradiction.  In particular, we will show that if the diamond configuration violates these bounds, then there must be a reachable configuration violating Lemma \ref{loopBounds}.  In particular, the diamond configuration must be such a configuration.

First, we must show that the diamond configuration is reachable.  Given a process $p$, define process $p'$ as follows: complete all firing moves in $p$ in the same order, but skip all diamond firing moves.  Then perform all diamond firing moves in the same order in which they occurred in $p$.  Since all diamond firing moves must occur after all non-diamond moves at neighboring sites, all of the moves in $p'$ are legal, and performing all of them in this order produces a configuration $s'$ after all of the non-diamond moves have been performed.  The number of firing moves that has occurred at any site $k$ is equal to $(m-|k|)^2-|k|$, so the resulting configuration $s'$ has 3 chips at the origin and 2 chips each at every other site from $-m+1$ to $m-1$.

We will show that this configuration $s'$ is actually equal to the diamond configuration $s_p$.  Any chip's position in $s'$ must also be the location of the first diamond move that it's involved in in $p'$, since the next firing move involving that chip must be a diamond move, and since all sites with chips must fire at least one more time.  Furthermore, since any diamond move in $p'$ must still occur after the same firing moves at the corresponding site and neighboring sites as in $p$, all diamond moves in $p'$ must contain the same chips as in $p$.  As a result, we get that $s_p=s_{p'}$, and since $s_{p'}$ is a reachable configuration, then $s_p$ must be as well.

Now, we can prove the desired chip bounds by contradiction.  Suppose that the number of chips with values than $k$ at positions greater than $l$ is at least $k+m-l$.  Then the diamond configuration must have a chip at position at least $l+\lceil\frac{k+m-l}{2}\rceil=\lceil\frac{k+m+l}{2}\rceil$ with value less than $k$.  If $l>0$, then a chip with value at most $k-1$ would have to reach position $\lceil\frac{k+m+2}{2}\rceil$, and if $l=0$, then two chips with values at most $k-1$ would have to reach position $\lceil \frac{k+m+1}{2}\rceil$.  Both of these are prohibited by Lemma \ref{loopBounds}.  Since the diamond configuration is a reachable configuration, it must satisfy Lemma \ref{loopBounds}, so this is a contradiction.  Thus, the diamond configuration may not have more than $k+m-l-1$ chips with values less than $k$ at positions greater than $l$, as desired.
\end{proof}

We can now move on to the main lemma for the self-loop labeled chip firing problem:

\begin{lemma}\label{mainlemma2}
Let $-m+1 \le k \le 0$ and $1 \le j \le m-|k|$.  After the $j^{th}$ diamond firing move at site $k$, there are at least $j+k+m-1$ chips with values less than $k$ at positions less than $k$.
\end{lemma}

\begin{proof}

We will induct first on $j$, and then on $k$.  For the case $k=-m+1$, there is only one firing move at that site, and there are no chips at smaller positions prior to that firing move.  Since the chip labeled $-m$ can never reach a position greater than 0, it must occupy a non-positive position in the diamond configuration, so it must be present for the first diamond move at some site with position at most 0 (corresponding to moves with $y$ values of $m$ in our labeling from \ref{mainlemma}.  Since it is the smallest chip, and since all diamond moves take place with exactly 3 chips present, it must then be fired to the left at every site it reaches until it is fired to the left at site $-m+1$ to its final position of $-m$.  Thus, after the 1 firing move at site $-m+1$, there is 1 chip with value less than $-m+1$ at site $-m$, satisfying the necessary conditions for $k=1$.

We then fix $k>-m+1$ and induct on $j$.  If we begin with the degenerate $j=0$, we have that before the first diamond move at site $k$, there are at most $k+m-1$ chips with values greater than or equal to $k$ at positions less than $k$ by \ref{lemma1}.  Since there are $2(k+m)-2$ total chips at positions less than $k$ prior to that move, at least $k+m-1$ of them must have values less than $k$.

Now, we suppose that our bounds hold for a particular value of $j$.  We then consider diamond move $j+1$ at site $k$.  We know that after move $j$ at this site, there are at least $j+k+m-1$ chips with values less than $k$ at positions less than $k$.  It is possible for this number to decrease by 1 after the $j^{th}$ firing move at site $k-1$, but this will only occur if all three of the chips at that site have values less than $k$.  Since there are at least $j+(k-1)+m-1$ chips with values less than $k-1$ at positions less than $k-1$, along with the chip that remains at position $k$ with value less than $k$ after that move occurs, this means that after the $j^{th}$ diamond move at site $k-1$, there will still be at least $j+k+m-1$ chips with values less than $k$ at positions less than $k$.

Now, from \ref{mainlemma2}, we have that when the first diamond move at site $j$ occurs, there are at most $k+m-j-1$ chips with values less than $k$ and positions greater than $j$.  Since there are $2(k+m)-1$ total chips with values less than $k$, there must be at least $k+m+j$ chips with values less than $k$ at positions less than or equal to $j$.  Thus, if there are not already $k+m+j$ chips with such values at positions less than $k$, then there must be at least one more at position less than or equal to $j$ when its first diamond move occurs.  If we use the labeling scheme from \ref{mainlemma}, then this implies that some chip with value less than $k$ will be involved in some firing move with $y=j-1$ and $x-y \ge k$.  If we consider the smallest label of all such chips, then that chip will be fired to the left by firing moves with $y=j-1$ until it is fired to the left by move $(k+j-1,j-1)$, which is the $j^{th}$ move at site $k$.  Since this chip has a label less than $k$, the number of chips with labels less than $k$ at positions less than $k$ will increase by $1$ if that number was previously equal to its minimum possible value of $k+m+j-1$.  As a result, this number is at least $k+m+j$ after the firing move occurs, and the induction is complete.

\end{proof}

\begin{center}
\begin{tikzpicture}

\draw[black, thick] (0,0) -- (4,4);
\draw[black, thick] (1,-1) -- (5,3);
\draw[black, thick] (2,-2) -- (6,2);
\draw[black, thick] (3,-3) -- (7,1);
\draw[black, thick] (4,-4) -- (8,0);
\draw[black, thick] (0,0) -- (4,-4);
\draw[black, thick] (1,1) -- (5,-3);
\draw[black, thick] (2,2) -- (6,-2);
\draw[black, thick] (3,3) -- (7,-1);
\draw[black, thick] (4,4) -- (8,0);

\filldraw[black] (0,0) circle (2pt) node[left] {$1\text{ }$};
\filldraw[black] (1,1) circle (2pt) node[left] {$2\text{ }$};
\filldraw[black] (1,-1) circle (2pt) node[left] {$3\text{ }$};
\filldraw[black] (2,2) circle (2pt) node[left] {$3\text{ }$};
\filldraw[black] (2,0) circle (2pt) node[left] {$4\text{ }$};
\filldraw[black] (2,-2) circle (2pt) node[left] {$5\text{ }$};
\filldraw[black] (3,3) circle (2pt) node[left] {$4\text{ }$};
\filldraw[black] (3,1) circle (2pt) node[left] {$5\text{ }$};
\filldraw[black] (3,-1) circle (2pt) node[left] {$6\text{ }$};
\filldraw[black] (3,-3) circle (2pt) node[left] {$7\text{ }$};
\filldraw[black] (4,4) circle (2pt) node[left] {$5\text{ }$} node[right] {$\text{ }5$};
\filldraw[black] (4,2) circle (2pt) node[left] {$6\text{ }$} node[right] {$\text{ }6$};
\filldraw[black] (4,0) circle (2pt) node[left] {$7\text{ }$} node[right] {$\text{ }7$};
\filldraw[black] (4,-2) circle (2pt) node[left] {$8\text{ }$} node[right] {$\text{ }8$};
\filldraw[black] (4,-4) circle (2pt) node[left] {$9\text{ }$} node[right] {$\text{ }9$};
\filldraw[black] (5,3) circle (2pt) node[right] {$\text{ }4$};
\filldraw[black] (5,1) circle (2pt) node[right] {$\text{ }5$};
\filldraw[black] (5,-1) circle (2pt) node[right] {$\text{ }6$};
\filldraw[black] (5,-3) circle (2pt) node[right] {$\text{ }7$};
\filldraw[black] (6,2) circle (2pt) node[right] {$\text{ }3$};
\filldraw[black] (6,0) circle (2pt) node[right] {$\text{ }4$};
\filldraw[black] (6,-2) circle (2pt) node[right] {$\text{ }5$};
\filldraw[black] (7,1) circle (2pt) node[right] {$\text{ }2$};
\filldraw[black] (7,-1) circle (2pt) node[right] {$\text{ }3$};
\filldraw[black] (8,0) circle (2pt) node[right] {$\text{ }1$};

\node at (4,5.5) {Sites};
\node at (-1,5) {$-5$};
\node at (0,5) {$-4$};
\node at (1,5) {$-3$};
\node at (2,5) {$-2$};
\node at (3,5) {$-1$};
\node at (4,5) {$0$};
\node at (5,5) {$1$};
\node at (6,5) {$2$};
\node at (7,5) {$3$};
\node at (8,5) {$4$};
\node at (9,5) {$5$};

\node[below] at (-1,-5) {\begin{tabular}{c} \(  \) \\ \( -5 \) \end{tabular}};
\node[below] at (0,-5) {\begin{tabular}{c} \( -4 \) \\ \( -4 \) \end{tabular}};
\node[below] at (1,-5) {\begin{tabular}{c} \( -3 \) \\ \( -3 \) \end{tabular}};
\node[below] at (2,-5) {\begin{tabular}{c} \( -2 \) \\ \( -2 \) \end{tabular}};
\node[below] at (3,-5) {\begin{tabular}{c} \( -1 \) \\ \( -1 \) \end{tabular}};
\node[below] at (4,-5) {\begin{tabular}{c} \(  \) \\ \( 0 \) \end{tabular}};
\node[below] at (5,-5) {\begin{tabular}{c} \( 1 \) \\ \( 1 \) \end{tabular}};
\node[below] at (6,-5) {\begin{tabular}{c} \( 2 \) \\ \( 2 \) \end{tabular}};
\node[below] at (7,-5) {\begin{tabular}{c} \( 3 \) \\ \( 3 \) \end{tabular}};
\node[below] at (8,-5) {\begin{tabular}{c} \( 4 \) \\ \( 4 \) \end{tabular}};
\node[below] at (9,-5) {\begin{tabular}{c} \(  \) \\ \( 5 \) \end{tabular}};

\draw[->, thick] (0,0) -- (-1,-5);
\draw[->, thick] (1,-1) -- (0,-5);
\draw[->, thick] (2,-2) -- (1,-5);
\draw[->, thick] (3,-3) -- (2,-5);
\draw[->, thick] (4,-4) -- (3,-5);
\draw[->, thick] (4,-4) -- (5,-5);
\draw[->, thick] (5,-3) -- (6,-5);
\draw[->, thick] (6,-2) -- (7,-5);
\draw[->, thick] (7,-1) -- (8,-5);
\draw[->, thick] (8,0) -- (9,-5);

\draw[->, thick] (0,0) -- (0,-5);
\draw[->, thick] (1,-1) -- (1,-5);
\draw[->, thick] (2,-2) -- (2,-5);
\draw[->, thick] (3,-3) -- (3,-5);
\draw[->, thick] (4,-4) -- (4,-5);
\draw[->, thick] (5,-3) -- (5,-5);
\draw[->, thick] (6,-2) -- (6,-5);
\draw[->, thick] (7,-1) -- (7,-5);
\draw[->, thick] (8,0) -- (8,-5);

\end{tikzpicture}


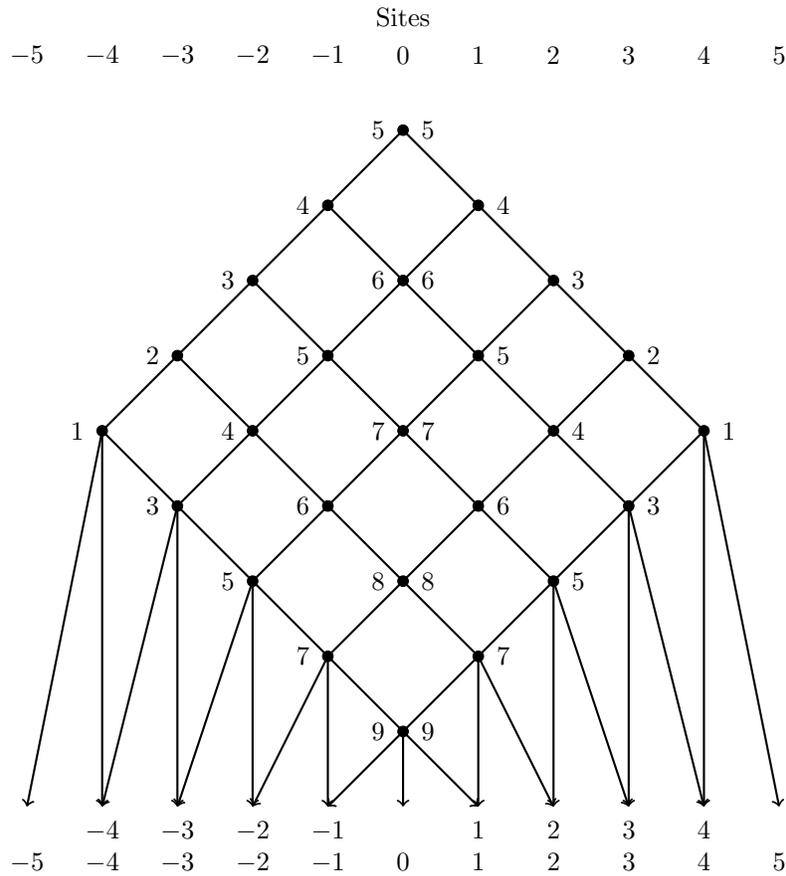
\captionof{figure}{Diagram for Lemma \ref{mainlemma2} and Lemma \ref{mainthm2}.  A number written to the left of a node is a lower bound on the number of chips with values less than the site number that must be to the left of the corresponding site when the firing move occurs.  The initial value at the top of the diamond is equal to $m-|k|$, where $k$ is the site number, and this minimum increases by 1 with each firing move at that site.  The arrows at the bottom show how the last firing moves at each site send each of the chips to their final, sorted positions.}

\end{center}

This brings us to our main result of this section:

\begin{theorem}\label{mainthm2}
In labeled chip firing with self loops, if the initial number of chips at the origin is congruent to 3 mod 4, then the chips end up in weakly sorted order.
\end{theorem}

\begin{proof}

By Lemma \ref{mainlemma2}, after the final firing move at position $k$ ($k<0$), there are $2(k+m)-1$ chips with values less than $k$ at positions less than $k$, with an analogous result holding for $k>0$.  Since the only way to satisfy this condition is to have all chips appear at positions equal to their labels, this means that the chips end up in weakly sorted order, as desired.

\end{proof}

\subsection{Self-Loops and Multiple Edges}

We now combine the cases of self-loops and multiple edges.  In this problem, each pair of adjacent nodes is connected by $r$ edges, and each node also has $r$ self-loops.  Each firing move now consists of choosing $3r$ chips at a given site, and then sending the $r$ smallest of those chips to the left and the $r$ largest to the right.  Conjecture 25 of \cite{HMP16} stated that all chips must end in weakly sorted order if the number of chips is divisible by $2r$.

This is essentially the same as the previous case in which everything has been scaled up by a factor of $r$.  In the same way that the multiple edges case extends the original problem, we can extend the self-loop problem here without too much extra work.  In particular, the firing move poset is the same, and the bounds that we were previously able to place on individual chips now apply to groups of $r$ chips.  With those changes in mind, the proofs of all relevant lemmas remain the same, and the same result is produced.

\begin{figure}
\begin{center}
\begin{tikzpicture}[scale=0.4]

\draw[black] (0,0) -- (0,-12);
\draw[black] (1,-3) -- (1,-11);
\draw[black] (-1,-3) -- (-1,-11);
\draw[black] (2,-8) -- (2,-11);
\draw[black] (-2,-8) -- (-2,-11);
\draw[black] (0,-2) -- (1,-3);
\draw[black] (0,-2) -- (-1,-3);
\draw[black] (0,-4) -- (1,-5);
\draw[black] (0,-4) -- (-1,-5);
\draw[black] (0,-6) -- (1,-7);
\draw[black] (0,-6) -- (-1,-7);
\draw[black] (0,-7) -- (1,-8);
\draw[black] (0,-7) -- (-1,-8);
\draw[black] (0,-9) -- (1,-10);
\draw[black] (0,-9) -- (-1,-10);
\draw[black] (0,-10) -- (1,-11);
\draw[black] (0,-10) -- (-1,-11);
\draw[black] (1,-7) -- (2,-8);
\draw[black] (-1,-7) -- (-2,-8);
\draw[black] (1,-10) -- (2,-11);
\draw[black] (-1,-10) -- (-2,-11);
\draw[black] (1,-5) -- (0,-9);
\draw[black] (-1,-5) -- (0,-9);
\draw[black] (1,-7) -- (0,-10);
\draw[black] (-1,-7) -- (0,-10);
\draw[black] (1,-8) -- (0,-11);
\draw[black] (-1,-8) -- (0,-11);
\draw[black] (2,-8) -- (1,-11);
\draw[black] (-2,-8) -- (-1,-11);
\draw[black] (1,-11) -- (0,-12);
\draw[black] (-1,-11) -- (0,-12);
\draw[black] (2,-11) -- (1,-13);
\draw[black] (-2,-11) -- (-1,-13);
\draw[black] (0,-12) -- (3,-15);
\draw[black] (-1,-13) -- (2,-16);
\draw[black] (-2,-14) -- (1,-17);
\draw[black] (-3,-15) -- (0,-18);
\draw[black] (0,-12) -- (-3,-15);
\draw[black] (1,-13) -- (-2,-16);
\draw[black] (2,-14) -- (-1,-17);
\draw[black] (3,-15) -- (0,-18);

\filldraw[black] (0,0) circle (3pt);
\filldraw[black] (0,-1) circle (3pt);
\filldraw[black] (0,-2) circle (3pt);
\filldraw[black] (0,-3) circle (3pt);
\filldraw[black] (0,-4) circle (3pt);
\filldraw[black] (0,-5) circle (3pt);
\filldraw[black] (0,-6) circle (3pt);
\filldraw[black] (0,-7) circle (3pt);
\filldraw[black] (0,-8) circle (3pt);
\filldraw[black] (0,-9) circle (3pt);
\filldraw[black] (0,-10) circle (3pt);
\filldraw[black] (0,-11) circle (3pt);
\filldraw[black] (0,-12) circle (3pt);
\filldraw[black] (0,-14) circle (3pt);
\filldraw[black] (0,-16) circle (3pt);
\filldraw[black] (0,-18) circle (3pt);
\filldraw[black] (1,-3) circle (3pt);
\filldraw[black] (1,-5) circle (3pt);
\filldraw[black] (1,-7) circle (3pt);
\filldraw[black] (1,-8) circle (3pt);
\filldraw[black] (1,-10) circle (3pt);
\filldraw[black] (1,-11) circle (3pt);
\filldraw[black] (1,-13) circle (3pt);
\filldraw[black] (1,-15) circle (3pt);
\filldraw[black] (1,-17) circle (3pt);
\filldraw[black] (-1,-3) circle (3pt);
\filldraw[black] (-1,-5) circle (3pt);
\filldraw[black] (-1,-7) circle (3pt);
\filldraw[black] (-1,-8) circle (3pt);
\filldraw[black] (-1,-10) circle (3pt);
\filldraw[black] (-1,-11) circle (3pt);
\filldraw[black] (-1,-13) circle (3pt);
\filldraw[black] (-1,-15) circle (3pt);
\filldraw[black] (-1,-17) circle (3pt);
\filldraw[black] (2,-8) circle (3pt);
\filldraw[black] (2,-11) circle (3pt);
\filldraw[black] (2,-14) circle (3pt);
\filldraw[black] (2,-16) circle (3pt);
\filldraw[black] (-2,-8) circle (3pt);
\filldraw[black] (-2,-11) circle (3pt);
\filldraw[black] (-2,-14) circle (3pt);
\filldraw[black] (-2,-16) circle (3pt);
\filldraw[black] (3,-15) circle (3pt);
\filldraw[black] (-3,-15) circle (3pt);

\draw[black] (10,0) -- (10,-1);
\draw[black] (10,-1) -- (11,-2);
\draw[black] (10,-1) -- (9,-2);
\draw[black] (9,-2) -- (12,-5);
\draw[black] (11,-2) -- (8,-5);
\draw[black] (9,-4) -- (13,-8);
\draw[black] (11,-4) -- (7,-8);
\draw[black] (8,-5) -- (14,-11);
\draw[black] (12,-5) -- (6,-11);
\draw[black] (8,-7) -- (13,-12);
\draw[black] (12,-7) -- (7,-12);
\draw[black] (7,-8) -- (12,-13);
\draw[black] (13,-8) -- (8,-13);
\draw[black] (7,-10) -- (11,-14);
\draw[black] (13,-10) -- (9,-14);
\draw[black] (6,-11) -- (10,-15);
\draw[black] (14,-11) -- (10,-15);

\filldraw[black](10,0) circle(3 pt);
\filldraw[black](10,-1) circle(3 pt);
\filldraw[black](10,-3) circle(3 pt);
\filldraw[black](10,-5) circle(3 pt);
\filldraw[black](10,-7) circle(3 pt);
\filldraw[black](10,-9) circle(3 pt);
\filldraw[black](10,-11) circle(3 pt);
\filldraw[black](10,-13) circle(3 pt);
\filldraw[black](10,-15) circle(3 pt);
\filldraw[black](11,-2) circle(3 pt);
\filldraw[black](11,-4) circle(3 pt);
\filldraw[black](11,-6) circle(3 pt);
\filldraw[black](11,-8) circle(3 pt);
\filldraw[black](11,-10) circle(3 pt);
\filldraw[black](11,-12) circle(3 pt);
\filldraw[black](11,-14) circle(3 pt);
\filldraw[black](9,-2) circle(3 pt);
\filldraw[black](9,-4) circle(3 pt);
\filldraw[black](9,-6) circle(3 pt);
\filldraw[black](9,-8) circle(3 pt);
\filldraw[black](9,-10) circle(3 pt);
\filldraw[black](9,-12) circle(3 pt);
\filldraw[black](9,-14) circle(3 pt);
\filldraw[black](12,-5) circle(3 pt);
\filldraw[black](12,-7) circle(3 pt);
\filldraw[black](12,-9) circle(3 pt);
\filldraw[black](12,-11) circle(3 pt);
\filldraw[black](12,-13) circle(3 pt);
\filldraw[black](8,-5) circle(3 pt);
\filldraw[black](8,-7) circle(3 pt);
\filldraw[black](8,-9) circle(3 pt);
\filldraw[black](8,-11) circle(3 pt);
\filldraw[black](8,-13) circle(3 pt);
\filldraw[black](13,-8) circle(3 pt);
\filldraw[black](13,-10) circle(3 pt);
\filldraw[black](13,-12) circle(3 pt);
\filldraw[black](7,-8) circle(3 pt);
\filldraw[black](7,-10) circle(3 pt);
\filldraw[black](7,-12) circle(3 pt);
\filldraw[black](14,-11) circle(3 pt);
\filldraw[black](6,-11) circle(3 pt);

\end{tikzpicture}
\caption{Left: the full firing order poset for $n=15$ in the self-loop problem.  Right: the firing order poset for the exponential edge case, with $t=3$ and $n=32$.} \label{figexp}
\end{center}
\end{figure}
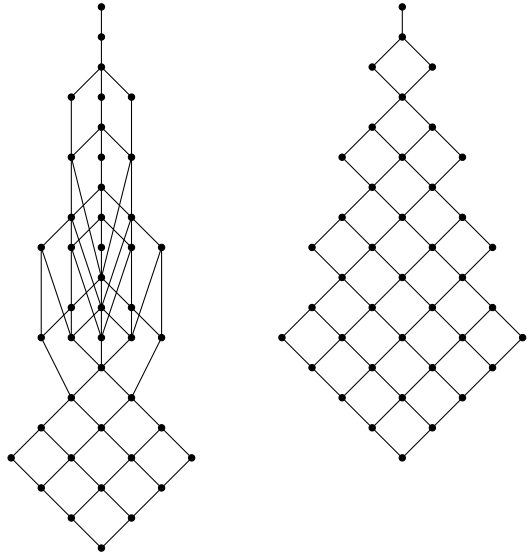

\section{Non-Sorting Cases}\label{non-sorting}

While much discussion has been devoted so far to a variety of settings in which labeled chip-firing can sort a collection of chips, this is not always the case.  Clearly, if we start with a single chip labeled $-1$ at position $1$, and a single chip labeled $1$ at position $-1$, then none of the chip-firing procedures discussed so far can possibly result in the sorting of this 2-chip configuration.  However, there are more subtle examples as well, with several non-sorting configurations that are very similar to the ones discussed so far.

\subsection{$1$-Dimensional Lattice, $n$ Odd}

The main result of this paper (and of \cite{HMP16}) was to prove that labeled chip-firing on a $1$-dimensional grid results in a sorted chip configuration where the number of chips $n$ is even.  When $n$ is odd, aside from the trivial $n=1$, this is no longer the case.

This distinction begins with the final unlabeled chip configuration.  With $2m$ chips, the resulting configuration has one chip each at sites $-m$ through $-1$, as well as at sites $1$ through $m$.  When a single chip is added, bringing the total to the odd $2m+1$, the only difference in the final configuration is that there is also a single chip at site 0.  Since the number of chips at every nonzero site remains the same, the number of firing moves required at each site does not change through the addition of this one chip.

\begin{center}
\begin{tikzpicture}[scale=0.7]

\draw[black, thick] (0,0) -- (0,-10);
\draw[black, thick] (1,-2) -- (1,-9);
\draw[black, thick] (-1,-2) -- (-1,-9);
\draw[black, thick] (2,-5) -- (2,-9);
\draw[black, thick] (-2,-5) -- (-2,-9);
\draw[black, thick] (0,-1) -- (1,-2);
\draw[black, thick] (0,-1) -- (-1,-2);
\draw[black, thick] (0,-3) -- (2,-5);
\draw[black, thick] (0,-4) -- (1,-5);
\draw[black, thick] (0,-3) -- (-2,-5);
\draw[black, thick] (0,-4) -- (-1,-5);
\draw[black, thick] (0,-6) -- (3,-9);
\draw[black, thick] (0,-7) -- (2,-9);
\draw[black, thick] (0,-8) -- (1,-9);
\draw[black, thick] (0,-6) -- (-3,-9);
\draw[black, thick] (0,-7) -- (-2,-9);
\draw[black, thick] (0,-8) -- (-1,-9);
\draw[black, thick] (0,-10) -- (4,-14);
\draw[black, thick] (-1,-11) -- (3,-15);
\draw[black, thick] (-2,-12) -- (2,-16);
\draw[black, thick] (-3,-13) -- (1,-17);
\draw[black, thick] (-4,-14) -- (0,-18);
\draw[black, thick] (0,-10) -- (-4,-14);
\draw[black, thick] (1,-11) -- (-3,-15);
\draw[black, thick] (2,-12) -- (-2,-16);
\draw[black, thick] (3,-13) -- (-1,-17);
\draw[black, thick] (4,-14) -- (0,-18);
\draw[black, thick] (1,-2) -- (0,-6);
\draw[black, thick] (1,-4) -- (0,-7);
\draw[black, thick] (1,-5) -- (0,-8);
\draw[black, thick] (1,-7) -- (0,-9);
\draw[black, thick] (1,-9) -- (0,-10);
\draw[black, thick] (-1,-2) -- (0,-6);
\draw[black, thick] (-1,-4) -- (0,-7);
\draw[black, thick] (-1,-5) -- (0,-8);
\draw[black, thick] (-1,-7) -- (0,-9);
\draw[black, thick] (-1,-9) -- (0,-10);
\draw[black, thick] (2,-8) -- (1,-9);
\draw[black, thick] (2,-9) -- (1,-11);
\draw[black, thick] (-2,-8) -- (-1,-9);
\draw[black, thick] (-2,-9) -- (-1,-11);
\draw[black, thick] (3,-9) -- (2,-12);
\draw[black, thick] (-3,-9) -- (-2,-12);

\filldraw[black] (0,0) circle (2pt) node[right] {$0_{15}$};
\filldraw[black] (0,-1) circle (2pt) node[right] {$0_{14}$};
\filldraw[black] (0,-2) circle (2pt) node[right] {$0_{13}$};
\filldraw[black] (0,-3) circle (2pt) node[right] {$0_{12}$};
\filldraw[black] (0,-4) circle (2pt) node[right] {$0_{11}$};
\filldraw[black] (0,-5) circle (2pt) node[right] {$0_{10}$};
\filldraw[black] (0,-6) circle (2pt) node[right] {$0_9$};
\filldraw[black] (0,-7) circle (2pt) node[right] {$0_8$};
\filldraw[black] (0,-8) circle (2pt) node[right] {$0_7$};
\filldraw[black] (0,-9) circle (2pt) node[right] {$0_6$};
\filldraw[black] (0,-10) circle (2pt) node[right] {$0_5$};
\filldraw[black] (0,-12) circle (2pt) node[right] {$0_4$};
\filldraw[black] (0,-14) circle (2pt) node[right] {$0_3$};
\filldraw[black] (0,-16) circle (2pt) node[right] {$0_2$};
\filldraw[black] (0,-18) circle (2pt) node[right] {$0_1$};
\filldraw[black] (1,-2) circle (2pt) node[right] {$1_{10}$};
\filldraw[black] (1,-4) circle (2pt) node[right] {$1_{9}$};
\filldraw[black] (1,-5) circle (2pt) node[right] {$1_{8}$};
\filldraw[black] (1,-7) circle (2pt) node[right] {$1_{7}$};
\filldraw[black] (1,-8) circle (2pt) node[right] {$1_{6}$};
\filldraw[black] (1,-9) circle (2pt) node[right] {$1_{5}$};
\filldraw[black] (1,-11) circle (2pt) node[right] {$1_{4}$};
\filldraw[black] (1,-13) circle (2pt) node[right] {$1_{3}$};
\filldraw[black] (1,-15) circle (2pt) node[right] {$1_{2}$};
\filldraw[black] (1,-17) circle (2pt) node[right] {$1_{1}$};
\filldraw[black] (-1,-2) circle (2pt) node[right] {$-1_{10}$};
\filldraw[black] (-1,-4) circle (2pt) node[right] {$-1_{9}$};
\filldraw[black] (-1,-5) circle (2pt) node[right] {$-1_{8}$};
\filldraw[black] (-1,-7) circle (2pt) node[right] {$-1_{7}$};
\filldraw[black] (-1,-8) circle (2pt) node[right] {$-1_{6}$};
\filldraw[black] (-1,-9) circle (2pt) node[right] {$-1_{5}$};
\filldraw[black] (-1,-11) circle (2pt) node[right] {$-1_{4}$};
\filldraw[black] (-1,-13) circle (2pt) node[right] {$-1_{3}$};
\filldraw[black] (-1,-15) circle (2pt) node[right] {$-1_{2}$};
\filldraw[black] (-1,-17) circle (2pt) node[right] {$-1_{1}$};
\filldraw[black] (2,-5) circle (2pt) node[right] {$2_{6}$};
\filldraw[black] (2,-8) circle (2pt) node[right] {$2_{5}$};
\filldraw[black] (2,-9) circle (2pt) node[right] {$2_{4}$};
\filldraw[black] (2,-12) circle (2pt) node[right] {$2_{3}$};
\filldraw[black] (2,-14) circle (2pt) node[right] {$2_{2}$};
\filldraw[black] (2,-16) circle (2pt) node[right] {$2_{1}$};
\filldraw[black] (-2,-5) circle (2pt) node[right] {$-2_{6}$};
\filldraw[black] (-2,-8) circle (2pt) node[right] {$-2_{5}$};
\filldraw[black] (-2,-9) circle (2pt) node[right] {$-2_{4}$};
\filldraw[black] (-2,-12) circle (2pt) node[right] {$-2_{3}$};
\filldraw[black] (-2,-14) circle (2pt) node[right] {$-2_{2}$};
\filldraw[black] (-2,-16) circle (2pt) node[right] {$-2_{1}$};
\filldraw[black] (3,-9) circle (2pt) node[right] {$3_{3}$};
\filldraw[black] (3,-13) circle (2pt) node[right] {$3_{2}$};
\filldraw[black] (3,-15) circle (2pt) node[right] {$3_{1}$};
\filldraw[black] (-3,-9) circle (2pt) node[right] {$-3_{3}$};
\filldraw[black] (-3,-13) circle (2pt) node[right] {$-3_{2}$};
\filldraw[black] (-3,-15) circle (2pt) node[right] {$-3_{1}$};
\filldraw[black] (4,-14) circle (2pt) node[right] {$4_{1}$};
\filldraw[black] (-4,-14) circle (2pt) node[right] {$-4_{1}$};

\draw[black, thick] (10,0) -- (10,-18);
\draw[black, thick] (11,-2) -- (11,-17);
\draw[black, thick] (9,-2) -- (9,-17);
\draw[black, thick] (12,-5) -- (12,-16);
\draw[black, thick] (8,-5) -- (8,-16);
\draw[black, thick] (13,-9) -- (13,-15);
\draw[black, thick] (7,-9) -- (7,-15);
\draw[black, thick] (10,-1) -- (11,-2);
\draw[black, thick] (10,-1) -- (9,-2);
\draw[black, thick] (10,-3) -- (12,-5);
\draw[black, thick] (10,-4) -- (11,-5);
\draw[black, thick] (10,-3) -- (8,-5);
\draw[black, thick] (10,-4) -- (9,-5);
\draw[black, thick] (10,-6) -- (13,-9);
\draw[black, thick] (10,-7) -- (12,-9);
\draw[black, thick] (10,-8) -- (11,-9);
\draw[black, thick] (10,-6) -- (7,-9);
\draw[black, thick] (10,-7) -- (8,-9);
\draw[black, thick] (10,-8) -- (9,-9);
\draw[black, thick] (10,-10) -- (14,-14);
\draw[black, thick] (10,-12) -- (13,-15);
\draw[black, thick] (10,-14) -- (12,-16);
\draw[black, thick] (10,-16) -- (11,-17);
\draw[black, thick] (10,-10) -- (6,-14);
\draw[black, thick] (10,-12) -- (7,-15);
\draw[black, thick] (10,-14) -- (8,-16);
\draw[black, thick] (10,-16) -- (9,-17);
\draw[black, thick] (11,-2) -- (10,-7);
\draw[black, thick] (11,-4) -- (10,-8);
\draw[black, thick] (11,-5) -- (10,-9);
\draw[black, thick] (11,-8) -- (10,-10);
\draw[black, thick] (11,-9) -- (10,-12);
\draw[black, thick] (11,-11) -- (10,-14);
\draw[black, thick] (11,-13) -- (10,-16);
\draw[black, thick] (11,-15) -- (10,-18);
\draw[black, thick] (9,-2) -- (10,-7);
\draw[black, thick] (9,-4) -- (10,-8);
\draw[black, thick] (9,-5) -- (10,-9);
\draw[black, thick] (9,-8) -- (10,-10);
\draw[black, thick] (9,-9) -- (10,-12);
\draw[black, thick] (9,-11) -- (10,-14);
\draw[black, thick] (9,-13) -- (10,-16);
\draw[black, thick] (9,-15) -- (10,-18);
\draw[black, thick] (12,-5) -- (11,-9);
\draw[black, thick] (12,-8) -- (11,-11);
\draw[black, thick] (12,-9) -- (11,-13);
\draw[black, thick] (12,-12) -- (11,-15);
\draw[black, thick] (12,-14) -- (11,-17);
\draw[black, thick] (8,-5) -- (9,-9);
\draw[black, thick] (8,-8) -- (9,-11);
\draw[black, thick] (8,-9) -- (9,-13);
\draw[black, thick] (8,-12) -- (9,-15);
\draw[black, thick] (8,-14) -- (9,-17);
\draw[black, thick] (13,-9) -- (12,-14);
\draw[black, thick] (13,-13) -- (12,-16);
\draw[black, thick] (7,-9) -- (8,-14);
\draw[black, thick] (7,-13) -- (8,-16);

\filldraw[black] (10,0) circle (2pt) node[right] {$0_{15}$};
\filldraw[black] (10,-1) circle (2pt) node[right] {$0_{14}$};
\filldraw[black] (10,-2) circle (2pt) node[right] {$0_{13}$};
\filldraw[black] (10,-3) circle (2pt) node[right] {$0_{12}$};
\filldraw[black] (10,-4) circle (2pt) node[right] {$0_{11}$};
\filldraw[black] (10,-5) circle (2pt) node[right] {$0_{10}$};
\filldraw[black] (10,-6) circle (2pt) node[right] {$0_9$};
\filldraw[black] (10,-7) circle (2pt) node[right] {$0_8$};
\filldraw[black] (10,-8) circle (2pt) node[right] {$0_7$};
\filldraw[black] (10,-9) circle (2pt) node[right] {$0_6$};
\filldraw[black] (10,-10) circle (2pt) node[right] {$0_5$};
\filldraw[black] (10,-12) circle (2pt) node[right] {$0_4$};
\filldraw[black] (10,-14) circle (2pt) node[right] {$0_3$};
\filldraw[black] (10,-16) circle (2pt) node[right] {$0_2$};
\filldraw[black] (10,-18) circle (2pt) node[right] {$0_1$};
\filldraw[black] (11,-2) circle (2pt) node[right] {$1_{10}$};
\filldraw[black] (11,-4) circle (2pt) node[right] {$1_{9}$};
\filldraw[black] (11,-5) circle (2pt) node[right] {$1_{8}$};
\filldraw[black] (11,-7) circle (2pt) node[right] {$1_{7}$};
\filldraw[black] (11,-8) circle (2pt) node[right] {$1_{6}$};
\filldraw[black] (11,-9) circle (2pt) node[right] {$1_{5}$};
\filldraw[black] (11,-11) circle (2pt) node[right] {$1_{4}$};
\filldraw[black] (11,-13) circle (2pt) node[right] {$1_{3}$};
\filldraw[black] (11,-15) circle (2pt) node[right] {$1_{2}$};
\filldraw[black] (11,-17) circle (2pt) node[right] {$1_{1}$};
\filldraw[black] (9,-2) circle (2pt) node[right] {$-1_{10}$};
\filldraw[black] (9,-4) circle (2pt) node[right] {$-1_{9}$};
\filldraw[black] (9,-5) circle (2pt) node[right] {$-1_{8}$};
\filldraw[black] (9,-7) circle (2pt) node[right] {$-1_{7}$};
\filldraw[black] (9,-8) circle (2pt) node[right] {$-1_{6}$};
\filldraw[black] (9,-9) circle (2pt) node[right] {$-1_{5}$};
\filldraw[black] (9,-11) circle (2pt) node[right] {$-1_{4}$};
\filldraw[black] (9,-13) circle (2pt) node[right] {$-1_{3}$};
\filldraw[black] (9,-15) circle (2pt) node[right] {$-1_{2}$};
\filldraw[black] (9,-17) circle (2pt) node[right] {$-1_{1}$};
\filldraw[black] (12,-5) circle (2pt) node[right] {$2_{6}$};
\filldraw[black] (12,-8) circle (2pt) node[right] {$2_{5}$};
\filldraw[black] (12,-9) circle (2pt) node[right] {$2_{4}$};
\filldraw[black] (12,-12) circle (2pt) node[right] {$2_{3}$};
\filldraw[black] (12,-14) circle (2pt) node[right] {$2_{2}$};
\filldraw[black] (12,-16) circle (2pt) node[right] {$2_{1}$};
\filldraw[black] (8,-5) circle (2pt) node[right] {$-2_{6}$};
\filldraw[black] (8,-8) circle (2pt) node[right] {$-2_{5}$};
\filldraw[black] (8,-9) circle (2pt) node[right] {$-2_{4}$};
\filldraw[black] (8,-12) circle (2pt) node[right] {$-2_{3}$};
\filldraw[black] (8,-14) circle (2pt) node[right] {$-2_{2}$};
\filldraw[black] (8,-16) circle (2pt) node[right] {$-2_{1}$};
\filldraw[black] (13,-9) circle (2pt) node[right] {$3_{3}$};
\filldraw[black] (13,-13) circle (2pt) node[right] {$3_{2}$};
\filldraw[black] (13,-15) circle (2pt) node[right] {$3_{1}$};
\filldraw[black] (7,-9) circle (2pt) node[right] {$-3_{3}$};
\filldraw[black] (7,-13) circle (2pt) node[right] {$-3_{2}$};
\filldraw[black] (7,-15) circle (2pt) node[right] {$-3_{1}$};
\filldraw[black] (14,-14) circle (2pt) node[right] {$4_{1}$};
\filldraw[black] (6,-14) circle (2pt) node[right] {$-4_{1}$};

\end{tikzpicture}


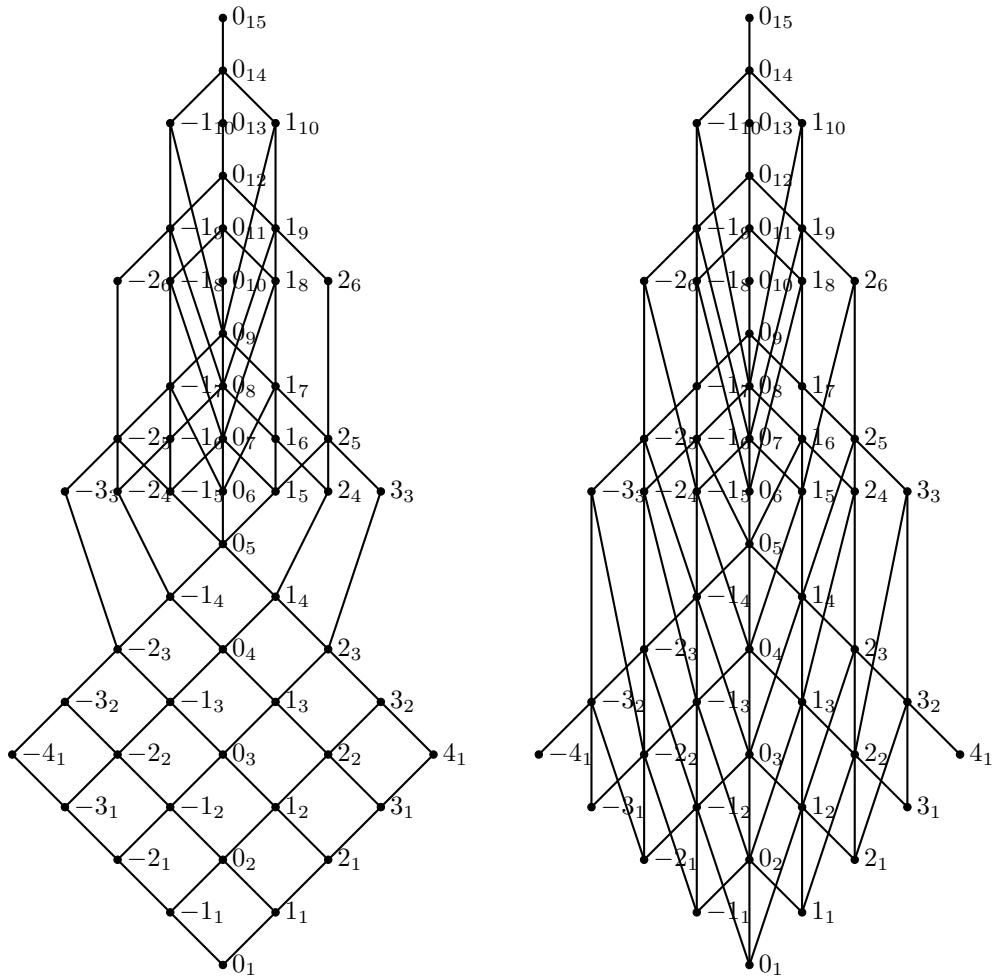
\captionof{figure}{The full firing order poset for $n=10$ (left) and $n=11$ (right).  These are identical to the Hasse diagrams for the $r$ edge case (for $10r \le n <11r$ and $11r \le n < 12r$, respectively) and for the origin self-loops case (for $n-s=10$ and $n-s=11$, respectively)}

\end{center}

The sorting proof breaks down almost immediately, as the last firing move at site 0 may now take place with up to 3 chips present.  This alone is enough to remove the guarantee of sorting.  Of the three ways to choose which two of the three chips to fire on the last move, only one choice will result in those three chips ending in sorted order.

This, however, is not the only way in which the change prevents sorting.  Since the last move at site $0$ can take place with three chips present, it can also take place earlier, before the last firing move at one of its neighbors.  This in turn breaks down the entire diamond structure, allowing any of the moves in the diamond to occur with three chips present (although not all in the same firing sequence).  In particular, the last firing move of the entire process must take place with three chips present.  This also allows the last firing move in the entire process to take place at any site that fires at least once, rather than only at the origin.  We compare the firing order posets for $n=10$ and $n=11$ below.

\begin{center}
\begin{tikzpicture}[scale=0.4]

\draw[black] (0,0) -- (0,-12);
\draw[black] (1,-3) -- (1,-11);
\draw[black] (-1,-3) -- (-1,-11);
\draw[black] (2,-8) -- (2,-11);
\draw[black] (-2,-8) -- (-2,-11);
\draw[black] (0,-2) -- (1,-3);
\draw[black] (0,-2) -- (-1,-3);
\draw[black] (0,-4) -- (1,-5);
\draw[black] (0,-4) -- (-1,-5);
\draw[black] (0,-6) -- (1,-7);
\draw[black] (0,-6) -- (-1,-7);
\draw[black] (0,-7) -- (1,-8);
\draw[black] (0,-7) -- (-1,-8);
\draw[black] (0,-9) -- (1,-10);
\draw[black] (0,-9) -- (-1,-10);
\draw[black] (0,-10) -- (1,-11);
\draw[black] (0,-10) -- (-1,-11);
\draw[black] (1,-7) -- (2,-8);
\draw[black] (-1,-7) -- (-2,-8);
\draw[black] (1,-10) -- (2,-11);
\draw[black] (-1,-10) -- (-2,-11);
\draw[black] (1,-5) -- (0,-9);
\draw[black] (-1,-5) -- (0,-9);
\draw[black] (1,-7) -- (0,-10);
\draw[black] (-1,-7) -- (0,-10);
\draw[black] (1,-8) -- (0,-11);
\draw[black] (-1,-8) -- (0,-11);
\draw[black] (2,-8) -- (1,-11);
\draw[black] (-2,-8) -- (-1,-11);
\draw[black] (1,-11) -- (0,-12);
\draw[black] (-1,-11) -- (0,-12);
\draw[black] (2,-11) -- (1,-13);
\draw[black] (-2,-11) -- (-1,-13);
\draw[black] (0,-12) -- (3,-15);
\draw[black] (-1,-13) -- (2,-16);
\draw[black] (-2,-14) -- (1,-17);
\draw[black] (-3,-15) -- (0,-18);
\draw[black] (0,-12) -- (-3,-15);
\draw[black] (1,-13) -- (-2,-16);
\draw[black] (2,-14) -- (-1,-17);
\draw[black] (3,-15) -- (0,-18);

\filldraw[black] (0,0) circle (3pt);
\filldraw[black] (0,-1) circle (3pt);
\filldraw[black] (0,-2) circle (3pt);
\filldraw[black] (0,-3) circle (3pt);
\filldraw[black] (0,-4) circle (3pt);
\filldraw[black] (0,-5) circle (3pt);
\filldraw[black] (0,-6) circle (3pt);
\filldraw[black] (0,-7) circle (3pt);
\filldraw[black] (0,-8) circle (3pt);
\filldraw[black] (0,-9) circle (3pt);
\filldraw[black] (0,-10) circle (3pt);
\filldraw[black] (0,-11) circle (3pt);
\filldraw[black] (0,-12) circle (3pt);
\filldraw[black] (0,-14) circle (3pt);
\filldraw[black] (0,-16) circle (3pt);
\filldraw[black] (0,-18) circle (3pt);
\filldraw[black] (1,-3) circle (3pt);
\filldraw[black] (1,-5) circle (3pt);
\filldraw[black] (1,-7) circle (3pt);
\filldraw[black] (1,-8) circle (3pt);
\filldraw[black] (1,-10) circle (3pt);
\filldraw[black] (1,-11) circle (3pt);
\filldraw[black] (1,-13) circle (3pt);
\filldraw[black] (1,-15) circle (3pt);
\filldraw[black] (1,-17) circle (3pt);
\filldraw[black] (-1,-3) circle (3pt);
\filldraw[black] (-1,-5) circle (3pt);
\filldraw[black] (-1,-7) circle (3pt);
\filldraw[black] (-1,-8) circle (3pt);
\filldraw[black] (-1,-10) circle (3pt);
\filldraw[black] (-1,-11) circle (3pt);
\filldraw[black] (-1,-13) circle (3pt);
\filldraw[black] (-1,-15) circle (3pt);
\filldraw[black] (-1,-17) circle (3pt);
\filldraw[black] (2,-8) circle (3pt);
\filldraw[black] (2,-11) circle (3pt);
\filldraw[black] (2,-14) circle (3pt);
\filldraw[black] (2,-16) circle (3pt);
\filldraw[black] (-2,-8) circle (3pt);
\filldraw[black] (-2,-11) circle (3pt);
\filldraw[black] (-2,-14) circle (3pt);
\filldraw[black] (-2,-16) circle (3pt);
\filldraw[black] (3,-15) circle (3pt);
\filldraw[black] (-3,-15) circle (3pt);

\draw[black] (8,0) -- (8,-18);
\draw[black] (9,-3) -- (9,-17);
\draw[black] (7,-3) -- (7,-17);
\draw[black] (10,-8) -- (10,-16);
\draw[black] (6,-8) -- (6,-16);
\draw[black] (8,-2) -- (9,-3);
\draw[black] (8,-2) -- (7,-3);
\draw[black] (8,-4) -- (9,-5);
\draw[black] (8,-4) -- (7,-5);
\draw[black] (8,-6) -- (9,-7);
\draw[black] (8,-6) -- (7,-7);
\draw[black] (8,-7) -- (9,-8);
\draw[black] (8,-7) -- (7,-8);
\draw[black] (8,-9) -- (9,-10);
\draw[black] (8,-9) -- (7,-10);
\draw[black] (8,-10) -- (9,-11);
\draw[black] (8,-10) -- (7,-11);
\draw[black] (9,-7) -- (10,-8);
\draw[black] (7,-7) -- (6,-8);
\draw[black] (9,-10) -- (10,-11);
\draw[black] (7,-10) -- (6,-11);
\draw[black] (8,-12) -- (11,-15);
\draw[black] (8,-12) -- (5,-15);
\draw[black] (8,-14) -- (10,-16);
\draw[black] (8,-14) -- (6,-16);
\draw[black] (8,-16) -- (9,-17);
\draw[black] (8,-16) -- (7,-17);
\draw[black] (9,-3) -- (8,-9);
\draw[black] (7,-3) -- (8,-9);
\draw[black] (9,-5) -- (8,-10);
\draw[black] (7,-5) -- (8,-10);
\draw[black] (9,-7) -- (8,-11);
\draw[black] (7,-7) -- (8,-11);
\draw[black] (9,-10) -- (8,-12);
\draw[black] (7,-10) -- (8,-12);
\draw[black] (9,-11) -- (8,-14);
\draw[black] (7,-11) -- (8,-14);
\draw[black] (9,-13) -- (8,-16);
\draw[black] (7,-13) -- (8,-16);
\draw[black] (9,-15) -- (8,-18);
\draw[black] (7,-15) -- (8,-18);
\draw[black] (10,-8) -- (9,-13);
\draw[black] (6,-8) -- (7,-13);
\draw[black] (10,-11) -- (9,-15);
\draw[black] (6,-11) -- (7,-15);
\draw[black] (10,-14) -- (9,-17);
\draw[black] (6,-14) -- (7,-17);

\filldraw[black] (8,0) circle (3pt);
\filldraw[black] (8,-1) circle (3pt);
\filldraw[black] (8,-2) circle (3pt);
\filldraw[black] (8,-3) circle (3pt);
\filldraw[black] (8,-4) circle (3pt);
\filldraw[black] (8,-5) circle (3pt);
\filldraw[black] (8,-6) circle (3pt);
\filldraw[black] (8,-7) circle (3pt);
\filldraw[black] (8,-8) circle (3pt);
\filldraw[black] (8,-9) circle (3pt);
\filldraw[black] (8,-10) circle (3pt);
\filldraw[black] (8,-11) circle (3pt);
\filldraw[black] (8,-12) circle (3pt);
\filldraw[black] (8,-14) circle (3pt);
\filldraw[black] (8,-16) circle (3pt);
\filldraw[black] (8,-18) circle (3pt);
\filldraw[black] (9,-3) circle (3pt);
\filldraw[black] (9,-5) circle (3pt);
\filldraw[black] (9,-7) circle (3pt);
\filldraw[black] (9,-8) circle (3pt);
\filldraw[black] (9,-10) circle (3pt);
\filldraw[black] (9,-11) circle (3pt);
\filldraw[black] (9,-13) circle (3pt);
\filldraw[black] (9,-15) circle (3pt);
\filldraw[black] (9,-17) circle (3pt);
\filldraw[black] (7,-3) circle (3pt);
\filldraw[black] (7,-5) circle (3pt);
\filldraw[black] (7,-7) circle (3pt);
\filldraw[black] (7,-8) circle (3pt);
\filldraw[black] (7,-10) circle (3pt);
\filldraw[black] (7,-11) circle (3pt);
\filldraw[black] (7,-13) circle (3pt);
\filldraw[black] (7,-15) circle (3pt);
\filldraw[black] (7,-17) circle (3pt);
\filldraw[black] (10,-8) circle (3pt);
\filldraw[black] (10,-11) circle (3pt);
\filldraw[black] (10,-14) circle (3pt);
\filldraw[black] (10,-16) circle (3pt);
\filldraw[black] (6,-8) circle (3pt);
\filldraw[black] (6,-11) circle (3pt);
\filldraw[black] (6,-14) circle (3pt);
\filldraw[black] (6,-16) circle (3pt);
\filldraw[black] (11,-15) circle (3pt);
\filldraw[black] (5,-15) circle (3pt);

\draw[black] (16,0) -- (16,-16);
\draw[black] (17,-3) -- (17,-15);
\draw[black] (15,-3) -- (15,-15);
\draw[black] (18,-8) -- (18,-15);
\draw[black] (14,-8) -- (14,-15);
\draw[black] (16,-2) -- (17,-3);
\draw[black] (16,-2) -- (15,-3);
\draw[black] (16,-4) -- (17,-5);
\draw[black] (16,-4) -- (15,-5);
\draw[black] (16,-6) -- (18,-8);
\draw[black] (16,-6) -- (14,-8);
\draw[black] (16,-7) -- (17,-8);
\draw[black] (16,-7) -- (15,-8);
\draw[black] (16,-9) -- (18,-11);
\draw[black] (16,-9) -- (14,-11);
\draw[black] (16,-10) -- (17,-11);
\draw[black] (16,-10) -- (15,-11);
\draw[black] (16,-12) -- (19,-15);
\draw[black] (16,-12) -- (13,-15);
\draw[black] (16,-13) -- (18,-15);
\draw[black] (16,-13) -- (14,-15);
\draw[black] (16,-14) -- (17,-15);
\draw[black] (16,-14) -- (15,-15);
\draw[black] (16,-16) -- (19,-19);
\draw[black] (15,-17) -- (18,-20);
\draw[black] (14,-18) -- (17,-21);
\draw[black] (13,-19) -- (16,-22);
\draw[black] (16,-16) -- (13,-19);
\draw[black] (17,-17) -- (14,-20);
\draw[black] (18,-18) -- (15,-21);
\draw[black] (19,-19) -- (16,-22);
\draw[black] (17,-3) -- (16,-10);
\draw[black] (15,-3) -- (16,-10);
\draw[black] (17,-5) -- (16,-11);
\draw[black] (15,-5) -- (16,-11);
\draw[black] (17,-8) -- (16,-12);
\draw[black] (15,-8) -- (16,-12);
\draw[black] (17,-10) -- (16,-13);
\draw[black] (15,-10) -- (16,-13);
\draw[black] (17,-11) -- (16,-14);
\draw[black] (15,-11) -- (16,-14);
\draw[black] (17,-13) -- (16,-15);
\draw[black] (15,-13) -- (16,-15);
\draw[black] (17,-15) -- (16,-16);
\draw[black] (15,-15) -- (16,-16);
\draw[black] (18,-8) -- (17,-14);
\draw[black] (14,-8) -- (15,-14);
\draw[black] (18,-14) -- (17,-15);
\draw[black] (14,-14) -- (15,-15);
\draw[black] (18,-15) -- (17,-17);
\draw[black] (14,-15) -- (15,-17);
\draw[black] (19,-15) -- (18,-18);
\draw[black] (13,-15) -- (14,-18);

\filldraw[black] (16,0) circle (3pt);
\filldraw[black] (16,-1) circle (3pt);
\filldraw[black] (16,-2) circle (3pt);
\filldraw[black] (16,-3) circle (3pt);
\filldraw[black] (16,-4) circle (3pt);
\filldraw[black] (16,-5) circle (3pt);
\filldraw[black] (16,-6) circle (3pt);
\filldraw[black] (16,-7) circle (3pt);
\filldraw[black] (16,-8) circle (3pt);
\filldraw[black] (16,-9) circle (3pt);
\filldraw[black] (16,-10) circle (3pt);
\filldraw[black] (16,-11) circle (3pt);
\filldraw[black] (16,-12) circle (3pt);
\filldraw[black] (16,-13) circle (3pt);
\filldraw[black] (16,-14) circle (3pt);
\filldraw[black] (16,-15) circle (3pt);
\filldraw[black] (16,-16) circle (3pt);
\filldraw[black] (16,-18) circle (3pt);
\filldraw[black] (16,-20) circle (3pt);
\filldraw[black] (16,-22) circle (3pt);
\filldraw[black] (17,-3) circle (3pt);
\filldraw[black] (17,-5) circle (3pt);
\filldraw[black] (17,-7) circle (3pt);
\filldraw[black] (17,-8) circle (3pt);
\filldraw[black] (17,-10) circle (3pt);
\filldraw[black] (17,-11) circle (3pt);
\filldraw[black] (17,-13) circle (3pt);
\filldraw[black] (17,-14) circle (3pt);
\filldraw[black] (17,-15) circle (3pt);
\filldraw[black] (17,-17) circle (3pt);
\filldraw[black] (17,-19) circle (3pt);
\filldraw[black] (17,-21) circle (3pt);
\filldraw[black] (15,-3) circle (3pt);
\filldraw[black] (15,-5) circle (3pt);
\filldraw[black] (15,-7) circle (3pt);
\filldraw[black] (15,-8) circle (3pt);
\filldraw[black] (15,-10) circle (3pt);
\filldraw[black] (15,-11) circle (3pt);
\filldraw[black] (15,-13) circle (3pt);
\filldraw[black] (15,-14) circle (3pt);
\filldraw[black] (15,-15) circle (3pt);
\filldraw[black] (15,-17) circle (3pt);
\filldraw[black] (15,-19) circle (3pt);
\filldraw[black] (15,-21) circle (3pt);
\filldraw[black] (18,-8) circle (3pt);
\filldraw[black] (18,-11) circle (3pt);
\filldraw[black] (18,-14) circle (3pt);
\filldraw[black] (18,-15) circle (3pt);
\filldraw[black] (18,-18) circle (3pt);
\filldraw[black] (18,-20) circle (3pt);
\filldraw[black] (14,-8) circle (3pt);
\filldraw[black] (14,-11) circle (3pt);
\filldraw[black] (14,-14) circle (3pt);
\filldraw[black] (14,-15) circle (3pt);
\filldraw[black] (14,-18) circle (3pt);
\filldraw[black] (14,-20) circle (3pt);
\filldraw[black] (19,-15) circle (3pt);
\filldraw[black] (19,-19) circle (3pt);
\filldraw[black] (13,-15) circle (3pt);
\filldraw[black] (13,-19) circle (3pt);

\draw[black] (24,0) -- (24,-22);
\draw[black] (25,-3) -- (25,-21);
\draw[black] (23,-3) -- (23,-21);
\draw[black] (26,-8) -- (26,-20);
\draw[black] (22,-8) -- (22,-20);
\draw[black] (27,-15) -- (27,-19);
\draw[black] (21,-15) -- (21,-19);
\draw[black] (24,-2) -- (25,-3);
\draw[black] (24,-2) -- (23,-3);
\draw[black] (24,-4) -- (25,-5);
\draw[black] (24,-4) -- (23,-5);
\draw[black] (24,-6) -- (26,-8);
\draw[black] (24,-6) -- (22,-8);
\draw[black] (24,-7) -- (25,-8);
\draw[black] (24,-7) -- (23,-8);
\draw[black] (24,-9) -- (26,-11);
\draw[black] (24,-9) -- (22,-11);
\draw[black] (24,-10) -- (25,-11);
\draw[black] (24,-10) -- (23,-11);
\draw[black] (24,-12) -- (27,-15);
\draw[black] (24,-12) -- (21,-15);
\draw[black] (24,-13) -- (26,-15);
\draw[black] (24,-13) -- (22,-15);
\draw[black] (24,-14) -- (25,-15);
\draw[black] (24,-14) -- (23,-15);
\draw[black] (24,-16) -- (27,-19);
\draw[black] (24,-16) -- (21,-19);
\draw[black] (24,-18) -- (26,-20);
\draw[black] (24,-18) -- (22,-20);
\draw[black] (24,-20) -- (25,-21);
\draw[black] (24,-20) -- (23,-21);
\draw[black] (25,-3) -- (24,-11);
\draw[black] (23,-3) -- (24,-11);
\draw[black] (25,-7) -- (24,-12);
\draw[black] (23,-7) -- (24,-12);
\draw[black] (25,-8) -- (24,-13);
\draw[black] (23,-8) -- (24,-13);
\draw[black] (25,-10) -- (24,-14);
\draw[black] (23,-10) -- (24,-14);
\draw[black] (25,-11) -- (24,-15);
\draw[black] (23,-11) -- (24,-15);
\draw[black] (25,-14) -- (24,-16);
\draw[black] (23,-14) -- (24,-16);
\draw[black] (25,-15) -- (24,-18);
\draw[black] (23,-15) -- (24,-18);
\draw[black] (25,-17) -- (24,-20);
\draw[black] (23,-17) -- (24,-20);
\draw[black] (25,-19) -- (24,-22);
\draw[black] (23,-19) -- (24,-22);
\draw[black] (26,-8) -- (25,-14);
\draw[black] (22,-8) -- (23,-14);
\draw[black] (26,-11) -- (25,-15);
\draw[black] (22,-11) -- (23,-15);
\draw[black] (26,-14) -- (25,-17);
\draw[black] (22,-14) -- (23,-17);
\draw[black] (26,-15) -- (25,-19);
\draw[black] (22,-15) -- (23,-19);
\draw[black] (26,-18) -- (25,-21);
\draw[black] (22,-18) -- (23,-21);
\draw[black] (27,-15) -- (26,-20);
\draw[black] (21,-15) -- (22,-20);

\filldraw[black] (24,0) circle (3pt);
\filldraw[black] (24,-1) circle (3pt);
\filldraw[black] (24,-2) circle (3pt);
\filldraw[black] (24,-3) circle (3pt);
\filldraw[black] (24,-4) circle (3pt);
\filldraw[black] (24,-5) circle (3pt);
\filldraw[black] (24,-6) circle (3pt);
\filldraw[black] (24,-7) circle (3pt);
\filldraw[black] (24,-8) circle (3pt);
\filldraw[black] (24,-9) circle (3pt);
\filldraw[black] (24,-10) circle (3pt);
\filldraw[black] (24,-11) circle (3pt);
\filldraw[black] (24,-12) circle (3pt);
\filldraw[black] (24,-13) circle (3pt);
\filldraw[black] (24,-14) circle (3pt);
\filldraw[black] (24,-15) circle (3pt);
\filldraw[black] (24,-16) circle (3pt);
\filldraw[black] (24,-18) circle (3pt);
\filldraw[black] (24,-20) circle (3pt);
\filldraw[black] (24,-22) circle (3pt);
\filldraw[black] (25,-3) circle (3pt);
\filldraw[black] (25,-5) circle (3pt);
\filldraw[black] (25,-7) circle (3pt);
\filldraw[black] (25,-8) circle (3pt);
\filldraw[black] (25,-10) circle (3pt);
\filldraw[black] (25,-11) circle (3pt);
\filldraw[black] (25,-13) circle (3pt);
\filldraw[black] (25,-14) circle (3pt);
\filldraw[black] (25,-15) circle (3pt);
\filldraw[black] (25,-17) circle (3pt);
\filldraw[black] (25,-19) circle (3pt);
\filldraw[black] (25,-21) circle (3pt);
\filldraw[black] (23,-3) circle (3pt);
\filldraw[black] (23,-5) circle (3pt);
\filldraw[black] (23,-7) circle (3pt);
\filldraw[black] (23,-8) circle (3pt);
\filldraw[black] (23,-10) circle (3pt);
\filldraw[black] (23,-11) circle (3pt);
\filldraw[black] (23,-13) circle (3pt);
\filldraw[black] (23,-14) circle (3pt);
\filldraw[black] (23,-15) circle (3pt);
\filldraw[black] (23,-17) circle (3pt);
\filldraw[black] (23,-19) circle (3pt);
\filldraw[black] (23,-21) circle (3pt);
\filldraw[black] (26,-8) circle (3pt);
\filldraw[black] (26,-11) circle (3pt);
\filldraw[black] (26,-14) circle (3pt);
\filldraw[black] (26,-15) circle (3pt);
\filldraw[black] (26,-18) circle (3pt);
\filldraw[black] (26,-20) circle (3pt);
\filldraw[black] (22,-8) circle (3pt);
\filldraw[black] (22,-11) circle (3pt);
\filldraw[black] (22,-14) circle (3pt);
\filldraw[black] (22,-15) circle (3pt);
\filldraw[black] (22,-18) circle (3pt);
\filldraw[black] (22,-20) circle (3pt);
\filldraw[black] (27,-15) circle (3pt);
\filldraw[black] (27,-19) circle (3pt);
\filldraw[black] (21,-15) circle (3pt);
\filldraw[black] (21,-19) circle (3pt);

\end{tikzpicture}


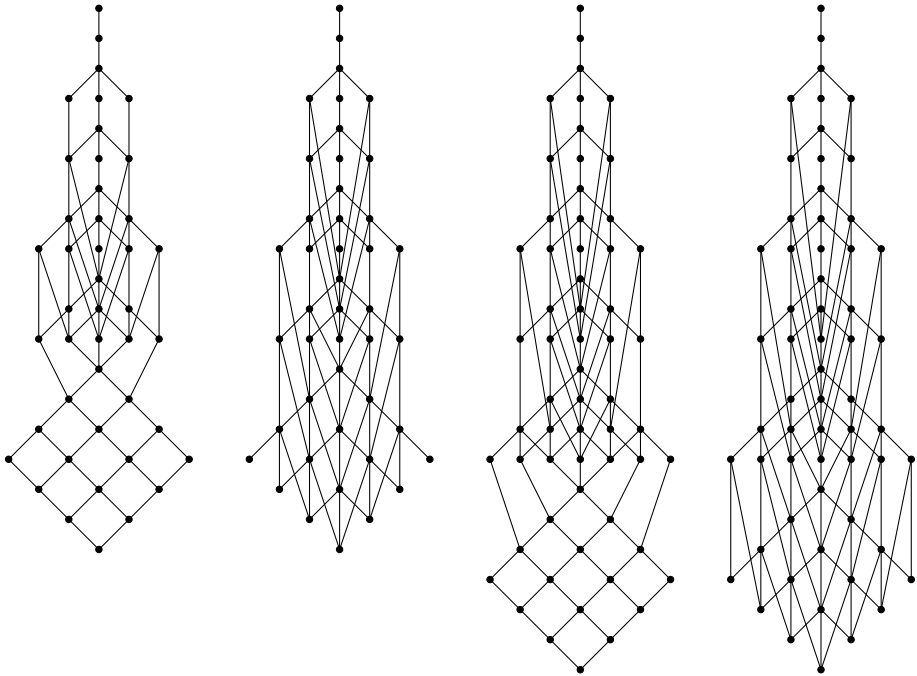
\captionof{figure}{The full firing order poset with one self-loop at each node for $n=$ $15$, $16$, $17$, and $18$.  Both the $n\equiv 1$ mod $4$ and $n\equiv 3$ mod $4$ cases have a diamond grid at the bottom, but the $n\equiv 1$ mod $4$ case has another equally large triangle above the diamond that can send chips past their sorted positions.}

\end{center}

\subsection{Self-Loop Graph, $n\equiv$ $1$ mod $4$}

It turns out that there can be distinctions that are even more subtle that can still prevent sorting from occurring.  If we consider the one dimensional grid with a single self-loop at each node, then for similar reasons as above, the chips might not sort when $n$ is even.  The last firing move may take place with 4 chips present, so that move may not sort those 4 chips if the wrong 3 are chosen to fire.

In the case where $n \equiv 1 \text{ mod } 4$, this issue does not arise.  In fact, the cases $n=4m-1$ and $n=4m+1$ both have identical diamonds at the end of their respective firing move posets, although the latter case does have additional moves that take place before then.  However, aside from $n=1$, sorting is never guaranteed when $n \equiv 1\text{ mod } 4$.  The Hasse diagrams for the full firing-order posets for $n=15$ through 18 are shown above.

Because this is so similar to the $n \equiv 3 \text{ mod } 4$ case, which does sort, this is in some sense a cutoff at which the diamond becomes ``too small'' to fix all of the out-of-order chips that it would need to fix.  There still appear to be relatively few unsorted configurations that can be reached, but such configurations nonetheless always exist for $n>1$.

To see this, we first consider the final unlabeled configuration in the $1 \text{ mod } 4$ case.  This will consist of 2 chips each at every site from $-m$ to $-1$ and $1$ to $m$ (where $n=4m+1$), and 1 chip at site 0.  This differs from the 1 mod 4 case in that the outermost sites now have 2 chips instead of 1.  If 1 mod 4 is the case where 1 extra chip ``spills over'' into each of two new positions, the 3 mod 4 case is where a second chip spills over to those positions.

Now, to reach an unsorted final configuration, we do the following.  Letting $n=4m+1$, we take both chips labeled $-m$ (since we now have 2 such chips instead of 1) and leave them at the origin as long as we can.  If we perform all other possible firing moves, then this is equivalent to the labeled chip-firing problem with only $4m-1$ chips, so all of these will end up in sorted order.  In particular, a chip with label $-m-1$ will end up at position $-m$.  We then allow the two chips labeled $-m$ to fire and run the process to completion.

Since site $-m$ never fires, this means that the chip labeled $-m-1$ will remain at site $-m$ until the end of the process, resulting in an unsorted configuration.  We can do this for any $m$, so it is always possible to avoid sorting if we choose our moves carefully for this case.

\section{Discussion}\label{discussion}

The earlier problems in which chip-firing methods were able to be used for sorting, along with some cases that don't, suggest the use of a pair of helpful conditions.  These two informal conditions, when taken together, are sufficient for proving global confluence in certain systems in which local confluence does not generally apply:

\begin{enumerate}

\item There is some collection of moves at the end of the process that do satisfy local confluence.
\item When the process enters this locally confluent region, all of the states that the system can be in will converge to the same final state when the process is run to completion.

\end{enumerate}

Note that if a system meets condition (1), then the state of the system when it enters the locally confluent region uniquely determines the final state.  Condition (2) requires additional restrictions on the process before that point in order to ensure that the unique final state is the same for every initial state in the confluent region.

For a wide range of problems, we were able to prove confluence by first establishing the existence of the diamond, which satisfies the first condition, and then by putting bounds on chip positions when they enter the diamond.  These bounds, combined with an analysis of what can happen in the diamond itself, provide us with the second condition, and thus with global confluence of the system.

\bibliographystyle{plain}

\bibliography{References}

\end{document}